\documentclass{amsart}
\usepackage[
top=30truemm,
bottom=30truemm,
left=30truemm,
right=30truemm]{geometry}
\usepackage{amsfonts, amssymb, amsmath, verbatim, amsthm, mathrsfs, amscd, enumerate, ascmac, fancyhdr, caption, hyperref, framed,ulem,subcaption,mathtools}
\usepackage{bbm}%mathbbm{1}
\usepackage{eucal}%fancy \mathcal
\usepackage{tikz}
\usetikzlibrary{shapes.geometric}
\usetikzlibrary{arrows.meta,arrows}
\usepackage{graphicx} % Required for inserting images
\usepackage{pgf, tikz}
\usepackage{pgfplots}
\pgfplotsset{compat=1.18}
\usepackage{float}
\usepackage{pdfpages}
\usepackage{tcolorbox}
\usepackage{thmtools}
\usepackage{enumitem}

\usepackage[framemethod=tikz]{mdframed}
\mdfsetup{linecolor=blue,backgroundcolor=gray!10,roundcorner=10pt,leftmargin=2pt,innerleftmargin=2pt,innerrightmargin=2pt}
\usepackage{url}

\usepackage{xcolor}
\usepackage{tikz}
\usetikzlibrary{positioning,intersections, calc, arrows,decorations.markings}

\newtheorem{theorem}{Theorem}[section]
\newtheorem{lemma}[theorem]{Lemma}
\newtheorem{proposition}[theorem]{Proposition}
\newtheorem*{theorem*}{Theorem}
\newtheorem{conjecture}{Conjecture}
\newtheorem*{conjecture*}{Conjecture}

\newtheorem*{question*}{Main question}

\newtheorem{corollary}[theorem]{Corollary}

\newtheorem*{claim*}{Claim}
\theoremstyle{definition}
\newtheorem{definition}[theorem]{Definition}
\newtheorem{example}[theorem]{Example}

\newtheorem*{goal*}{Goal}

\numberwithin{equation}{section}
\numberwithin{theorem}{section}

\newtheorem{remark}{Remark}

\numberwithin{remark}{section}

\DeclareMathOperator{\sinc}{sinc}

\usepackage{refcount}   % to get the number of a labeled theorem

\newenvironment{restatetheorem}[1]{%
  \par\medskip\noindent
  \textbf{Theorem \getrefnumber{#1} (in detail).}\itshape
}{%
  \par\medskip\normalfont
}

\usepackage{stackengine,scalerel}

\makeatletter
\def\@setauthors{%
  \begingroup
  \def\thanks{\protect\thanks@warning}%
  \trivlist
  \centering\footnotesize \@topsep30\p@\relax
  \advance\@topsep by -\baselineskip
  \item\relax
  \author@andify\authors
  \def\\{\protect\linebreak}

  \normalsize\lowercase{\authors}%
  
	\ifx\@empty\contribs
  \else
    ,\penalty-3 \space \@setcontribs
    \@closetoccontribs
  \fi
  \endtrivlist
  \endgroup
}
\def\@settitle{\begin{center}
\LARGE\lowercase{\@title}
  \end{center}%
}
\makeatother
\newcommand{\authoraddress}[1]{\address{\normalfont{#1}}}

\stackMath
\newcommand\reallywidehat[1]{%
\savestack{\tmpbox}{\stretchto{%
  \scaleto{%
    \scalerel*[\widthof{\ensuremath{#1}}]{\kern.1pt\mathchar"0362\kern.1pt}%
    {\rule{0ex}{\textheight}}
  }{\textheight}% 
}{2.4ex}}%
\stackon[-6.9pt]{#1}{\tmpbox}%
}
\makeatletter
\DeclareRobustCommand\widecheck[1]{{\mathpalette\@widecheck{#1}}}
\def\@widecheck#1#2{%
    \setbox\z@\hbox{\m@th$#1#2$}%
    \setbox\tw@\hbox{\m@th$#1%
       \widehat{%
          \vrule\@width\z@\@height\ht\z@
          \vrule\@height\z@\@width\wd\z@}$}%
    \dp\tw@-\ht\z@
    \@tempdima\ht\z@ \advance\@tempdima2\ht\tw@ \divide\@tempdima\thr@@
    \setbox\tw@\hbox{%
       \raise\@tempdima\hbox{\scalebox{1}[-1]{\lower\@tempdima\box
\tw@}}}%
    {\ooalign{\box\tw@ \cr \box\z@}}}
\makeatother

\def\eqalign#1{\null\,\vcenter{\openup\jot\mathsurround\dimen12
  \ialign{\strut\hfil$\textstyle{##}$&$\textstyle{{}##}$\hfil
      \crcr#1\crcr}}\,}

\def\supp{{\rm supp}}

\def\d{{\rm d}}

\def\R{\mathcal{R}}

\def\rd{\mathbb{R}^d}

\newcommand{\hd}{\dim_{\mathrm{H}}}
\newcommand{\fd}{\dim_{\mathrm{F}}}

\newcommand{\pd}{\dim_{\mathrm{P}}}
\newcommand{\lbd}{\underline{\dim}_{\textup{B}}}
\newcommand{\ubd}{\overline{\dim}_{\textup{B}}}

\title{Knapp-type obstructions in multilinear fractal Fourier extension}
\author{Itamar Oliveira}
\authoraddress{I. Oliveira, School of Mathematics, The Watson Building, University of Birmingham, Edgbaston, Birmingham, B15 2TT, England.}
\email{i.oliveira@bham.ac.uk, oliveira.itamar.w@gmail.com}
\thanks{IO was financially supported by EPSRC Grant EP/W032880/1. He is currently supported by his EPSRC Fellowship UKRI3285 \textit{New perspectives in phase-space Analysis and Fourier restriction}}

\author{Ana E. de Orellana}
\authoraddress{A. E. de Orellana, School of Mathematics and Statistics, University of St Andrews, St Andrews, KY16 9SS, Scotland.}
\email{aedo1@st-andrews.ac.uk}
\thanks{AEdO was financially supported by the University of St Andrews.}

\usepackage{fancyhdr}
\begin{document}

\begin{abstract}
% For curved smooth hypersurfaces, the classical Knapp example shows that the Stein--Tomas theorem about linear Fourier restriction estimates is sharp. Still in the smooth setting, variants of this example combined with the geometric notion of \textit{transversality} motivate the $L^{2}$-based multilinear Fourier extension conjecture.

% The classical Stein--Tomas extends to fractal measures exhibiting Fourier decay, as shown by .... It was only in the work of Hambrook and {\L}aba that the first cases of this extension were shown to be sharp by means of a `fractal' Knapp example involving Cantor-type constructions. Our main results are multilinear Knapp-type construction inspired by the works of Hambrook--{\L}aba and Chen for fractal measures; we obtain two necessary conditions for a multilinear fractal Fourier extension theorem to hold: one in terms of the upper box dimension of the measures' supports, and another one in terms of their Fourier decay and a ball condition. In particular, these conditions give a more restrictive range compared with previously known results whenever the convolution of the measures at play is singular.

% To complement our main theorem with results in the positive direction, we establish a multilinear Fourier extension estimate for measures whose convolution belongs to an $L^p$ space, which provides a rich class of examples of `transversal' self-similar measures through the work of Shmerkin and Solomyak.

For curved, smooth hypersurfaces, the classical Knapp example shows that the Stein--Tomas theorem, which gives linear Fourier restriction estimates, is sharp. Variants of this example combined with the geometric notion of \textit{transversality} motivate the $L^{2}$-based multilinear Fourier extension conjecture. In the fractal setting, work by Mockenhaupt, Mitsis, and Bak-Seeger extended the linear Fourier restriction estimate beyond the smooth setting, and subsequent work showed this extension to be sharp. In this article, we construct multilinear Knapp-type examples for fractal measures inspired by the works of Hambrook--{\L}aba and Chen. This yields two necessary conditions for a fractal multilinear Fourier extension estimate to hold: one in terms of the upper box dimension of the measures' supports, and another in terms of their Fourier decay and a ball condition. These conditions give a more restrictive range compared with previously known results whenever the convolution of the underlying measures is singular. In contrast, we complement this with a result in the positive direction by establishing a multilinear Fourier extension estimate for measures whose convolution lies in an $L^p$ space. This provides a rich class of examples of `transversal' self-similar measures through the work of Shmerkin and Solomyak.
\end{abstract}

\maketitle

\tableofcontents

\section{Introduction}

\subsection{Linear theory: smooth setting versus fractal setting} Given a smooth hypersurface $S\subseteq\mathbb{R}^{d}$ equipped with a measure $\sigma$, the \textit{linear Fourier extension problem} asks for which pairs $(p,q)$ one has
\begin{equation}\label{ineq1-050925}
    \| \widehat{g\d\sigma}\|_{L^{q}(\mathbb{R}^{d})}\lesssim \|g\|_{L^{p}(\mathrm{d}\sigma)},
\end{equation}
where 
\begin{equation}
     \widehat{g\d\sigma}(\xi)=\int_{S}g(x)e^{-2\pi ix\cdot\xi}\mathrm{d}\sigma(x)
\end{equation}
is the \textit{Fourier extension operator} associated to $\sigma$. This problem lies at the centre of a constellation of deep questions in Analysis, from the study of certain Fourier summability methods to Geometric Measure Theory and nonlinear dispersive PDEs. We refer the reader to the classical survey \cite{Tao-notes} for a more in-depth account of this problem (see also \cite{Dem,Mat15}). The range of exponents $(p,q)$ for which \eqref{ineq1-050925} holds is the content of \textit{Stein's restriction (extension) conjecture} (cf. Chapter IX of \cite{Steinbook2}):

\begin{conjecture}\label{restriction} If $S$ has nonvanishing Gaussian curvature, the inequality
\begin{equation}\label{rest1}
 \| \widehat{g\mathrm{d}\sigma}\|_{L^{q}(\mathbb{R}^{d})}\lesssim \|g\|_{L^{p}(\d\sigma)}
\end{equation}
holds if and only if $q>\frac{2d}{d-1}$ and $q\geq \frac{(d+1)}{d-1}p^{\prime}$, where $\sigma$ is the surface measure.
\end{conjecture}

Asymptotics of Bessel functions, combined with the nonvanishing Gaussian curvature hypothesis, imply a decay rate of $|\xi|^{-\frac{(d-1)}{2}}$ for $\widehat{\mathrm{d}\sigma}(\xi)$ and that $\widehat{\mathrm{d}\sigma}\notin L^{q}(\mathbb{R}^{d})$ if $q\leq\frac{2d}{d-1}$, hence the necessary condition $q>\frac{2d}{d-1}$ follows simply by plugging in $g\equiv 1$ on \eqref{rest1}. The other condition $q\geq \frac{(d+1)}{d-1}p^{\prime}$ follows by the classical \textit{Knapp example} (see \cite{Dem}). Conjecture \ref{restriction} has only been settled in the $d=2$ case (\cite{Fef1,Zyg}), but significant progress has been made in higher dimensions over the last five decades (see \cite{Stri1,Tomas,Bourg1,Guth1,WangWu} and the references therein).

The statement of Conjecture \ref{restriction} hints that \textit{curvature} has great influence in the behaviour of $\widehat{\mathrm{d}\sigma}$; indeed, if $S=\mathbb{R}\times\{0\}\subseteq\mathbb{R}^{2}$, then
$\widehat{\mathrm{d}\sigma}(x,y)$ is constant in $y$ and hence $\|\widehat{\mathrm{d}\sigma}\|_{L^{q}(\mathbb{R}^{d})}<\infty\Longrightarrow q=\infty$. We shall generically refer to an estimate such as
$$|\widehat{\mathrm{d}\sigma}(\xi)|\lesssim|\xi|^{-\rho},\qquad \rho>0,$$
as a \textit{Fourier decay rate} for the measure $\mathrm{d}\sigma$. The previous paragraph then leads us to the classical heuristics: \textit{curvature implies Fourier decay}.

In fractal geometry, the distinction between Fourier decay and lack thereof is essential. While `random' fractals behave more like curved manifolds due to their Fourier decay (see e.g. \cite{Sal51,FdO26}), arithmetically structured sets, like the middle third Cantor set often lack measures with Fourier decay, thus behaving more like flat surfaces (see \cite{LP22},\cite[Chapter~8]{Mat15}).

This way of capturing curvature in terms of decay of the Fourier transform is what lies at the heart of the following version of the Stein--Tomas restriction theorem.
\begin{theorem}[Mockenhaupt \cite{Moc00}, Mitsis \cite{Mit02}, Bak--Seeger \cite{BS11}]\label{thm:ST} Let $\mu$ be a finite, compactly supported, Borel measure on $\rd$ such that for some $\alpha,\beta>0$,
    \begin{equation*}
    \mu(B(x,r)) \lesssim r^{\alpha}
\end{equation*}
for all $x\in\rd$ and $r>0$, and
\begin{equation*}
    \big| \widehat{\d\mu}(\xi) \big|^{2} \lesssim |\xi|^{-\beta}
\end{equation*}
for all $\xi\in\rd$. Then for all $f\in L^2(\mu)$ and $q\geq 2+ 4\frac{d-\alpha}{\beta}$,
\begin{equation}\label{eq:ST}
    \|\widehat{f\d\mu}\|_{L^q(\rd)} \lesssim \|f\|_{L^2(\d\mu)}.
\end{equation}
\end{theorem}
The origin of this theorem dates back to the works of Stein (see Chapter VIII of \cite{Steinbook2}) and Tomas \cite{Tomas} in the special case where $\mu$ is the surface measure of the sphere $\mathbb{S}^{d-1}$. On the other hand, it was not until \cite{Moc00,Mit02} that the problem was considered specifically for fractal measures (see \cite{BS11,CFdO} for a more thorough account of the history of this problem and recent progress). To the best of our knowledge, the only known restriction estimates beyond $L^2$ are those obtained by interpolating \eqref{eq:ST} with the trivial $L^1\to L^\infty$, and the one obtained in \cite{Chen}, where it is assumed that the measure has a convolution power in some $L^p$ space.

The conditions in Conjecture~\ref{restriction} are known to be sharp by the Knapp example, where the idea is to capture the `flatness' of $S$ by considering the characteristic function of a small cap. In \cite{HL,HL16,FHR,Che16} it was proven that the range of Theorem~\ref{thm:ST} is also sharp when considering fractal measures (see Section~\ref{sec:MultiKnappV3}). What is more, in \cite{HL} the authors noted that \eqref{ineq1-050925} could not hold for any $q<\frac{2d}{D_{2}(\mu)}$, where $D_{2}(\mu)$ is the $L^2$ dimension of the measure $\mu$, see Section~\ref{sec:prelim}. Generalising their remark, a more restrictive condition was obtained for general measures in \cite{CFdO} by considering a family of dimensions called the Fourier spectrum. In Proposition~\ref{prop:topLid} we use a similar argument to the one in \cite{HL} together with dimensional bounds to give a necessary condition for multilinear Fourier extension estimates to hold for general measures.

\subsection{Multilinear theory: smooth setting versus fractal setting} Multilinear variants of Conjecture~\ref{restriction} arose naturally in Klainerman and Machedon's work on well-posedness of certain PDEs (see \cite{KM1,KM2,KM3}), and their impact in the study of Conjecture \ref{restriction} subsequently evidenced in \cite{TVV1}. On the other hand, the idea of taking advantage of some type of `multilinear gain' to approach problems in the realm of linear Fourier extension theory predates \cite{TVV1}; it appears in the works of Prestini \cite{Prestini}, Drury \cite{Drury} and Christ \cite{Christ} in the context of Fourier extension for curves in $\mathbb{R}^{n}$, and also in the works by Fefferman-Stein \cite{Fef1}, C\'ordoba \cite{Cordoba}, Carleson-Sj\"olin \cite{CarlesonSjolin} and Zygmund \cite{Zyg}.

Given $k\geq 2$ smooth hypersurfaces $S_{j}$ equipped with measures $\sigma_{j}$, $1\leq j\leq k$, the \textit{multilinear Fourier extension problem} asks for which tuples of exponents $(p_{1},\ldots,p_{k},q)$ one has

\begin{equation}\label{ineq1-060925}
    \left\|\prod_{j=1}^{k}\widehat{g_{j}\mathrm{d}\sigma_{j}}\right\|_{L^{q}(\mathbb{R}^{d})}\lesssim \prod_{j=1}^{k}\|g_{j}\|_{L^{p_{j}}(\d\sigma_{j})}.
\end{equation}
For simplicity, the statement `\textit{\eqref{ineq1-060925} holds for all $g_{j}\in L^{p_{j}}(\sigma_{j})$}' will be represented by the notation $\R^{\ast}_{\sigma_{1},\ldots,\sigma_{k}}(p_{1}\times\cdots\times p_{k}\rightarrow q)$.

A straightforward application of H\"older's inequality implies a certain range of estimates such as \eqref{ineq1-060925} under the assumption of Conjecture \ref{restriction}, and this argument gives the optimal answer to the previous question if no further assumptions on the $S_j$ are made. 

Exploiting geometric features of the set of hypersurfaces $S_j$ other than curvature often leads to estimates of the form \eqref{ineq1-060925} that are \textit{beyond} those implied by \eqref{rest1}. A very natural such feature is the following: we say that $k$ smooth hypersurfaces $S_{1},\ldots,S_{k}\subseteq\mathbb{R}^{d}$ are \textit{transversal} if, for some $c>0$,
\begin{equation}\label{deftrans}
    |v_{1}\wedge\ldots\wedge v_{k}|\geq c 
\end{equation}
for all choices $v_{1},\ldots,v_{k}$ of unit normal vectors to $S_{1},\ldots,S_{k}$, respectively. In other words, if the $k$-dimensional volume of the parallelepiped generated by $v_{1},\ldots,v_{k}$ is bounded below by some absolute constant for any choice of normal vectors $v_{j}$, then the hypersurfaces are transversal. Under this condition, the \textit{multilinear Fourier extension conjecture} predicts what estimates of the form \eqref{ineq1-060925} should hold (in the case $p_{i}=p_{j}$ for all $1\leq i,j\leq k$):
\begin{conjecture}[\cite{Benn1}]\label{generalklinearapr422} Let $k\geq 2$ and suppose that $S_1,\ldots,S_k\subseteq\mathbb{R}^{d}$ equipped with measures $\sigma_{1},\ldots,\sigma_{k}$, respectively, are transversal with everywhere positive principal curvatures. If $\frac{1}{q}<\frac{d-1}{2d}$, $\frac{1}{q}\leq\frac{d+k-2}{d+k}\frac{1}{p^{\prime}}$ and $\frac{1}{q}\leq\frac{d-k}{d+k}\frac{1}{p^{\prime}}+\frac{k-1}{k+d}$, then $\mathcal{R}^{\ast}_{\sigma_{1},\ldots,\sigma_{k}}(p\times\cdots\times p\rightarrow q/k)$.
\end{conjecture}
The curvature hypotheses in Conjecture \ref{generalklinearapr422} can be removed in some cases; for instance, when $k=d$ (see \cite{Benn1,MOl} for a detailed exposition on the literature and recent progress). Conjecture \ref{generalklinearapr422} was settled in three cases:
\begin{enumerate}[label=(\roman*)]
\item $k=2$ was proved by Tao in \cite{Tao1} up to the endpoint.
\item $k=d$ was settled by Bennett, Carbery and Tao, up to the endpoint, in \cite{BCT}.
\item $k=d-1$ was established by Bejenaru in \cite{Bej} up to the endpoint.
\end{enumerate}

\begin{remark}\label{rem2} A clear contrast between \eqref{ineq1-050925} and \eqref{ineq1-060925} is already seen when $k=d=2$; if $S_1=\mathbb{R}\times\{0\}\subseteq\mathbb{R}^2$ and $S_2=\{0\}\times\mathbb{R}\subseteq\mathbb{R}^2$ are the coordinate axes of $\mathbb{R}^2$ (equipped with the induced Lebesgue measure $\sigma_{1}$ and $\sigma_{2}$, respectively) then \eqref{ineq1-050925} can only hold (individually for each measure) if $q=\infty$ due to lack of curvature. On the other hand, it is straightforward to check that, since $S_1$ and $S_2$ are transversal in the sense of the definition above,
$$\|\widehat{g_{1}\mathrm{d}\sigma_{1}}\widehat{g_{2}\mathrm{d}\sigma_{2}}\|_{L^{2}(\mathbb{R}^{2})}\lesssim \|g_1\|_{L^{2}(\d\sigma_1)}\|g_2\|_{L^{2}(\d\sigma_2)},$$
i.e. $\mathcal{R}^{\ast}_{\sigma_{1},\sigma_{2}}(2\times 2\rightarrow 2)$ holds. The latter estimate is \textbf{not} obtainable directly through individual bounds on $\|\widehat{g_{j}\mathrm{d}\sigma_{j}}\|_{L^q(\mathbb{R}^2)}$, $j=1,2$, $\forall q\geq 1$.
\end{remark}

Remark \ref{rem2} raises a very natural question that is at the core of this manuscript. 

\begin{question*} In the fractal setting, under what conditions do multilinear Fourier extension estimates hold \textbf{beyond} those implied by linear ones?
\end{question*}

As evidenced by Remark \ref{rem2}, transversality plays a fundamental role when the underlying measures are supported in smooth manifolds. Measures with well-separated Fourier support give rise to wave packets with limited overlaps, leading to a gain unavailable in the linear theory. However, that geometric notion does not seem to have an immediate analogue in the fractal world (see \cite{Ferrante} for a related discussion). As we will see, Proposition \ref{thm:mainthm} indicates that $L^{p}$-regularity of the $k$-fold convolution $\mu_{1}\ast\cdots\ast\mu_{k}$ allows one to prove multilinear estimates beyond the span of linear ones, but our main results Theorems \ref{nec-sing} and \ref{HL-nec-sing} suggest that `fractal transversality' could still manifest itself even when $\mu_{1}\ast\cdots\ast\mu_{k}$ is singular. Example \ref{example1-22072026} exhibits a related phenomenon in the context of Fourier extension for smooth curves in $\mathbb{R}^{3}$.

By the Hausdorff--Young inequality, \eqref{ineq1-060925} can hold for $q\leq 2$ only if the convolution of the measures at play is absolutely continuous and with an $L^{q'}$ density (this imposes $q>2$ in the singular convolution case, see Figure \ref{figure:singularrange}). This is not usually a concern for hypersurfaces in the smooth setting: if $\sigma_{1}$ and $\sigma_{2}$ are two disjoint arcs of a smooth convex curve, then one can show that $\sigma_{1}\ast\sigma_{2}\ll\mathcal{L}^{2}$ and $\sigma_{1}\ast\sigma_{2}\in L^{\infty}(\mathbb{R}^{2})$ (see \cite{OeS14}), which is enough to prove bilinear Fourier extension estimates for $(f,g)\mapsto \widehat{f\mathrm{d}\sigma_{1}}\widehat{g\mathrm{d}\sigma_{2}}$ beyond those following from linear ones. Notice that these arcs are transversal in the sense of \eqref{deftrans}, therefore bilinear estimates of that kind follow from \cite{BCT}, for instance. In some cases one even has, for smooth manifolds, a \textit{characterisation} of some multilinear estimates in terms of absolute continuity of the underlying measures; see Theorem 1.1 of \cite{Bennett-Bez}. In general, given two smooth and transversal hypersurfaces $S_{1},S_2\subseteq \mathbb{R}^d$ equipped with measures $\sigma_{1}$ and $\sigma_{2}$, respectively, $\sigma_{1}\ast\sigma_{2}\ll\mathcal{L}^{d}$. That is, hypersurface transversality already accounts for the convolution being absolutely continuous. 

The question of absolute continuity of the convolution in the fractal setting is more delicate; if $\hd (\supp{(\mu_{1})}+\supp{(\mu_{2})})<1$, then $\mu_{1}\ast\mu_{2}$ is singular with respect to the Lebesgue measure, whereas finding conditions that guarantee $\mu_{1}\ast\mu_{2}\ll\mathcal{L}^{d}$ is difficult in general (see \cite{SS}). The canonical example of a fractal measure is the natural self-similar measure $\mu_r$ in the middle $(1 - 2r)$-Cantor set for some $0<r<\frac{1}{2}$. Studying the interaction between these measures in physical space suggests that a reasonable notion of `transversality' for two Cantor measures $\mu_{r_1}$ and $\mu_{r_2}$ could be the requirement that $r_1$ and $r_2$ are rationally independent, i.e. that they satisfy $\frac{\log r_1}{\log r_2}\notin\mathbb{Q}$; this heuristic, however, is not supported by the available evidence. As pointed out in \cite{NPS12}, there is a dense $G_\delta$ set of parameters $u\in\mathbb{R}$ such that, with $T_u(x) = ux$, $\mu_{1/3}*T_u\mu_{1/4}$ is singular despite the fact that $\frac{\log3}{\log4}\notin\mathbb{Q}$. Examples like this one show that defining a suitable notion of transversality in the fractal setting will need to take into account subtleties that are not present in classical Fourier restriction theory. Note that in this article we refer to `transversality' as the geometric condition defined in \eqref{deftrans}, which, to the best of our knowledge, is not related to transversality methods in fractal geometry, see e.g. \cite{BSS}.

The main objective of this paper is, in a few words, to present necessary conditions for an estimate such as $\R^{\ast}_{\mu_{1},\ldots,\mu_{k}}(p\times\cdots\times p\rightarrow q)$ to hold for measures $\mu_{1},\ldots,\mu_{k}$ under \textit{size} and \textit{Fourier decay} hypotheses (see Theorem~\ref{nec-sing}). %Additionally, we present a simple sufficient condition in higher dimensions that highlights the regularity of the $k$-fold convolution $\mu_{1}\ast\cdots\ast\mu_{k}$.

Building on the work of Hambrook--{\L}aba and Chen \cite{HL, Che16}, in Theorem~\ref{nec-sing} we give a necessary condition for multilinear Fourier extension estimates to hold in dimension $1$. The result we obtain here is more restrictive than \cite[Proposition 5.3]{Trainor} in many cases of interest, for instance whenever the sum of the Hausdorff dimensions of the supports of the measures involved is less than $1$, and it could be complemented by a singular convolution version of Proposition~\ref{thm:mainthm}. As shown in examples \ref{ex:estimateSingular} and \ref{example1-22072026}, multilinear Fourier extension estimates still hold for measures with singular convolution. We believe that it would be interesting to further explore the singular case and to obtain a wider variety of examples of multilinear estimate that do not follow directly from linear ones.

As a modest complement to our main result, we present a sufficient condition on the $k$-fold convolution $\mu_{1}\ast\cdots\ast\mu_{k}$ for which multilinear Fourier extension estimates hold. More precisely, in Proposition \ref{thm:mainthm} we show that such estimates hold whenever $\mu_{1}*\cdots*\mu_{k}$ is absolutely continuous with respect to $\mathcal{L}^{d}$ and is in some Lebesgue space. This slightly extends the result of \cite{Trainor} where the author considered the stronger assumption of an $L^{\infty}$ density.

In \cite[Example~5.2]{Trainor} the author exhibits a pair of self-similar measures with equal contraction ratios whose digit sets satisfy a condition that implies that their convolution is the Lebesgue measure in $[0,1]$. This naturally raises the question of whether weaker arithmetic or combinatorial conditions between the digit sets might imply that the convolution lies in some $L^p$ space for $p\neq\infty$. Rather than pursuing digit set conditions directly, we rely on the fractal geometry literature to give examples of measures that satisfy the conditions of Proposition~\ref{thm:mainthm} depending on their dimensional and Fourier analytic properties. These results allow us to give examples of measures for which multilinear Fourier extension estimates hold depending on the $L^q$ and Fourier dimensions of the measures involved (see Section~\ref{sec:prelim}).
\\
\\
\textbf{Structure of the paper.} Section \ref{sec:prelim} is dedicated to the preliminaries on fractal geometry and dimension theory. These concepts will be used to give examples of Proposition~\ref{thm:mainthm} via Corollaries~\ref{cor1} and \ref{cor2}, and counterexamples that give the necessary conditions for multilinear Fourier extension estimates to hold; see Theorems~\ref{nec-sing} and \ref{HL-nec-sing}. Section \ref{mainresults-310126} contains the statements of the main results, examples and remarks. In Section \ref{proof-main-abs-310126} we prove Proposition~\ref{thm:mainthm}. Finally, in Section~\ref{sec:MultiKnappV3} we prove Theorems~\ref{nec-sing} and \ref{HL-nec-sing}.
\\
\\
\textbf{Acknowledgments.} We thank Jonathan Bennett and Jonathan M. Fraser for many stimulating conversations, remarks, and suggestions during the preparation of this manuscript.

\section{Preliminaries}\label{sec:prelim}

In this section, we fix the notation and introduce concepts from Fractal geometry that will be used later in the paper.
\\
\\
\textbf{Notation.} Throughout the paper we write $A\lesssim B$ to indicate that there exists a universal constant $0<C<\infty$ such that $A\leq CB$. By $A\approx B$ we mean that both $A\lesssim B$ and $B\lesssim A$ hold. We write $[N]$, with $N\in\mathbb{N}$, for the set $\{0,\ldots, N-1\}$. and $|A|$ for the cardinality of a finite set $A$. Finally, for $p,q\in[1,\infty]$ we write $p',q'$ to refer to their H\"older conjugate exponents.
\\
\\
\textbf{Iterated Function Systems.} An iterated function system (IFS) on $\mathbb{R}$ is a finite collection $\boldsymbol\varphi = (\varphi_i)_{i=1}^m$ of contractive maps. We will only consider the case where these maps are linear similarities $\varphi_i = \lambda_ix +a_i$, for $\lambda_i,a_i\in \mathbb{R}$ and $|\lambda_i|<1$. When all $\lambda_i$ are equal we say that the IFS is homogeneous. It is well-known (see e.g. \cite[Theorem 9.1]{falconer}) that for any IFS $\boldsymbol\varphi$ there exists a unique non-empty compact set $K\subseteq\mathbb{R}$, known as the attractor (or self-similar set), such that
\begin{equation*}
    K = \bigcup_{i=1}^m \varphi_i(K).
\end{equation*}
We say that an IFS with attractor $K$ satisfies the strong separation condition if $\varphi_i(K)\cap\varphi_j(K)=\varnothing$ whenever $i\neq j$, and it satisfies the open set condition if there exists a non-empty open set $U$ such that $\cup_{i=1}^m \varphi_i(U)\subseteq U$ with the union being disjoint. Given a probability vector $(p_i)_{i=1}^m$, the self-similar measure determined by the IFS $\boldsymbol{\varphi}$ and $(p_i)_{i=1}^m$ is the unique Borel probability measure $\mu$ on $K$ satisfying
\begin{equation}\label{eq:IFSmeasure}
    \mu = \sum_{i=1}^{m} p_i\,\varphi_i\mu.
\end{equation}
In the homogeneous case and when all probabilities are equal $p_i = \frac{1}{m}$, we refer to $\mu$ as the natural self-similar measure of the attractor. For a more thorough presentation on iterated function systems and their dimensions we refer the reader to \cite{falconer}.
\\
\\
\textbf{Dimensions of Fractal Sets and Measures.} Throughout this article we will work with non-zero, finite, compactly supported, Borel measures on $\rd$.

Given $0<\alpha\leq d$, we say that a measure $\mu$ satisfies the Frostman condition with exponent $\alpha$ if for any $x\in\rd$ and $r>0$, $\mu(B(x,r))\lesssim r^\alpha$. This leads us to our first notion of fractal dimension of measures, the $L^\infty$-dimension, also known as Frostman dimension in the literature,
\begin{equation*}
    D_\infty(\mu) =\sup\big\{ \alpha\in[0,d] : \mu( B(x,r) )\lesssim r^\alpha,~~\forall r>0,x\in\rd \big\}.
\end{equation*}
As a consequence of Frostman's lemma (see e.g. \cite[Theorem~2.7]{Mat15}), if $D_\infty(\mu)>\alpha$, then for any $s<\alpha$ the $s$-energy
\begin{equation*}
    I_s(\mu) \coloneqq \iint \frac{\d\mu(x)\d\mu(y)}{|x-y|^s}
\end{equation*}
is finite. These energies play an important role in fractal geometry. They give rise to the definition of the $L^2$-dimension, which quantifies how a measure distributes mass on average
\begin{equation*}
    D_2(\mu) = \sup\{ s\in[0,d] : I_s(\mu)<\infty \}.
\end{equation*}
By Frostman's lemma, it is not difficult to see that these dimensions characterise the Hausdorff dimension of a set. Given $X\subseteq\rd$
\begin{equation*}
    \hd X = \sup\{ D_2(\mu) : \mu\text{ on }X \}.
\end{equation*}

A natural generalisation of the dimensions defined above is the $L^q$-dimensions, which capture the differences in the distribution of mass on the attractor. These dimensions are of particular interest for measures that exhibit multifractal behaviour, and are known for many particular examples. We refer the reader to \cite{Ols95} for more details. For $q>1$, the $L^q$-dimension of a measure $\mu$ is defined by
\begin{equation*}
    D_q(\mu) = \sup\Bigg\{s\in [0,d] : \int\bigg( \int  \frac{\d\mu(y)}{|x-y|^{s}}\bigg)^{q-1}\,\d\mu(x) <\infty \Bigg\}.
\end{equation*}
The $L^q$-dimensions were first introduced in the following `lower' form. For $q\geq0$ and $r>0$ let
\begin{equation*}
    C_{q}(\mu,r) = \int \mu\big( B(x,r) \big)^{q-1}\,d\mu(x).
\end{equation*}
Then for $q\neq 1$ the $L^q$-dimension of $\mu$ is
\begin{equation*}
    D_{q}(\mu) = \liminf_{r\to0} \frac{\log C_{q}(\mu,r)}{(q-1)\log r},
\end{equation*}
and for $q=1$,
\begin{equation*}
    D_{1}(\mu) = \liminf_{r\to0} \frac{\int \log\mu\big( B(x,r) \big)\,d\mu(x)}{\log r}.
\end{equation*}
The map $q\mapsto D_q(\mu)$ is non-increasing for $q\geq0$, and continuous except perhaps at $q=1$. The $L^q$-dimensions include many well known notions of dimension, e.g. $\lim_{q\to\infty} D_q(\mu)= D_\infty(\mu)$, and $D_0(\mu) = \lbd\supp(\mu)$, where $\lbd X$ is the lower box dimension of $X\subseteq\rd$, also defined as
 \begin{equation*}
     \lbd X = \liminf_{\delta\searrow0}\frac{\log N_\delta(X)}{-\log\delta},
 \end{equation*}
 where $N_\delta(X)$ is the smallest number of sets of diameter $\delta$ needed to cover $X$. The upper box dimension $\ubd X$ is defined likewise by changing  the above $\liminf$ for a $\limsup$.

For measures on homogeneous IFSs defined as in \eqref{eq:IFSmeasure} satisfying the open set condition, the above formula can be simplified to
\begin{equation*}
    D_q(\mu) =\frac{\log \sum_{i=1}^m p_i^q}{(q-1)\log \lambda}
\end{equation*}
for $q\neq1$, and 
\begin{equation*}
    D_1(\mu) = \frac{\sum_{i=1}^m p_i\log p_i}{\log \lambda}.
\end{equation*}

We will also make use of the packing dimension of a Borel set $X\subseteq\rd$, defined as
\begin{equation*}
    \pd X = \inf\left\{ \sup_i \ubd U_i : X\subseteq\bigcup_{i=1}^\infty U_i, \text{ $U_i$ bounded}\right\},
\end{equation*}
these dimensions satisfy $\hd X \leq \pd X \leq \ubd X$ and for any $\alpha$-H\"older map $f$ and $X\subseteq\rd$, $\pd f(X)\leq \alpha^{-1}\pd (X)$.

For $0<s< d$, it is easy to see by Parseval's theorem and the fact that as a distribution $\widehat{|\cdot|^{-s}}(\xi) = |\xi|^{s-d}$, that
\begin{equation*}
    I_s(\mu) \approx_{s,d} \int_{\rd} \big|\widehat{\d\mu}(\xi)\big|^2 |\xi|^{s-d} \,\d\xi
\end{equation*}
Since we are only considering finite measures, the region of integration in the above can be restricted to $|\xi|>1$. Thus, if $\widehat{\d\mu}$ decays polynomially, i.e. if $|\widehat{\d\mu}(\xi)|\lesssim |\xi|^{-t/2}$ for some $t>s$, $I_s(\mu)<\infty$. This motivates the definition of the Fourier dimension of a measure
\begin{equation*}
    \fd\mu = \sup\Big\{ s\in[0,d] : \sup_{\xi\in\rd} \big|\widehat{\d\mu}(\xi)\big|^2|\xi|^{s} <\infty \Big\},
\end{equation*}
and of a Borel set $X\subseteq\rd$
\begin{equation*}
  \fd X = \sup\{ s\in[0,d] : \exists \mu\text{ on }X : \fd\mu\geq s\}.
\end{equation*}

For a Borel set $X\subseteq\rd$, $\fd X\leq \hd X$. Sets that satisfy $\hd X = \fd X$ are called Salem sets. Also, in \cite{Mit02} the author showed that for any Borel measure $\mu$ on $\rd$, $\frac{\fd\mu}{2}\leq D_\infty(\mu)$.

\section{Main Results}\label{mainresults-310126}

Our main results are Knapp-type constructions that give necessary conditions for a fractal multilinear Fourier extension theorem to hold. In this section, we will present these constructions and complement them with a modest sufficient condition that generates a large number of interesting examples.

\noindent\textbf{A multilinear Knapp example for Cantor-type measures.} We establish a necessary condition for multilinear Fourier extension estimates to hold (see Figure~\ref{figure:singularrange}). The underlying construction should be viewed as a multilinear analogue of the fractal Knapp examples from \cite{HL,Che16}. We construct measures whose mass concentrates along transversal configurations (in some sense that will be made precise later) that favour multilinear gains, which is the new technical feature that we introduce (see Section~\ref{sec:MultiKnappV3}).
\begin{theorem}\label{nec-sing} Let $1 < p < \infty$ and $1\leq q<\infty$. Let $0<\beta_m\leq \alpha_m<1$ for $m=1,\ldots,k$ such that
    \begin{equation}\label{cond1-280126}
    \alpha_{j+1}-\frac{\beta_{j+1}}{2}\leq\alpha_{j}-\frac{\beta_{j}}{2},\quad\textnormal{for all} \quad 1\leq j\leq k-1,
\end{equation}
and
\begin{equation}\label{cond2-280126}
    \left(\alpha_{k}-\frac{\beta_{k}}{2}\right)+(k-1)\left(\alpha_{1}-\frac{\beta_{1}}{2}\right)<1.
\end{equation}
There exist $k$ Borel measures $\mu_{m}$ supported on compact sets $E_m$ for $m=1,\ldots, k$ satisfying
    \begin{equation}\label{Fdecay-270126}
        |\widehat{\d\mu_m}(\xi)|\lesssim |\xi|^{-\beta_m/2},
    \end{equation}
    and that for every $\varepsilon>0$, there is $0<r_{\varepsilon}<\frac{1}{4}$ with
    \begin{equation}\label{eq:ahlfors}
        r^{\alpha_m+\varepsilon}\lesssim \mu_{m}\big(B(x,r)\big) \lesssim r^{\alpha_m},\quad\forall\quad 0<r<r_{\varepsilon},\quad x\in E_{m},
    \end{equation}
    such that if
    \begin{equation*}
        q<\frac{2p(1-\sum_{m=1}^k\alpha_m) + p\sum_{m=1}^k\beta_m}{(p-1)\sum_{m=1}^{k}\beta_m},
    \end{equation*}
    then there exist $k$ sequences of functions $\{f_{\ell,m}\}_{\ell\in\mathbb{N}}$, $m=1,\ldots, k$, such that
    \begin{equation*}
    \frac{\|\prod_{m=1}^{k}\widehat{f_{\ell,m} \d\mu_m}\|_{L^q(\mathbb{R})}}{\prod_{m=1}^{k}\|f_{\ell,m}\|_{L^p(\d\mu_m)}} \to\infty \text{ as }\ell\to\infty.
    \end{equation*}
\end{theorem}

The following result differs from Theorem~\ref{nec-sing} in that the measures constructed are supported on Salem sets, and requires the dimensions $\alpha_m$ to be of the form $\alpha_m = \log _Nt$ for $t,N\in\mathbb{N}$, however, it relaxes \eqref{eq:ahlfors} to a Frostman condition.
\begin{theorem}\label{HL-nec-sing} Let $1 < p < \infty$ and $1\leq q<\infty$. For each $m=1,\ldots,k$ let $0<\alpha_k<\cdots <\alpha_{1}<1$ such that $\alpha_m = \frac{\log t_{0,m}}{\log N_0}$ for some $t_{0,m},N_0\in\mathbb{N}$, and $(k-1)\alpha_1 + \alpha_k <2$. There exist $k$ Borel measures $\mu_{m}$ supported on compact sets $E_m$ for $m=1,\ldots,k$ satisfying
    \begin{equation*}
        |\widehat{\d\mu_m}(\xi)|\lesssim |\xi|^{-\beta_m/2}
    \end{equation*}
    for every $\beta_m<\alpha_m$, and
    \begin{equation*}
        \mu_m\big(B(x,r)\big) \lesssim r^{\alpha_m}
    \end{equation*}
    for $x\in E_m$ and $r>0$, such that if 
    \begin{equation*}
        q<\frac{p(2-\sum_{m=1}^k\alpha_m)}{(p-1)\sum_{m=1}^{k}\alpha_m},
    \end{equation*}
    then there exist $k$ sequences of functions $\{f_{\ell,m}\}_{\ell\in\mathbb{N}}$, $m=1,\ldots,k$, such that
    \begin{equation*}
    \frac{\|\prod_{m=1}^{k} \reallywidehat{f_{\ell,m}\d\mu_m}\|_{L^q(\mathbb{R})}}{\prod_{m=1}^{k}\|f_{\ell,m}\|_{L^p(\d\mu_m)}} \to\infty \text{ as }\ell\to\infty.
    \end{equation*}
\end{theorem}

We write more precise versions of Theorems~\ref{nec-sing} and \ref{HL-nec-sing} in Section~\ref{sec:proofProp}, where we go into more detail about what the measures $\mu_m$ are, and we replace \eqref{eq:ahlfors} by a weaker condition.

The following result generalises the necessary condition from \cite{HL} from linear to multilinear estimates.
\begin{proposition}\label{prop:topLid}
    Let $\mu_1,\ldots,\mu_k$ be finite, compactly supported, Borel measures on $\rd$. If the estimate
    \begin{equation*}
        \Big\| \widehat{f_1\d\mu_1}\cdots\widehat{f_k\d\mu_k} \Big\|_{L^q(\rd)} \lesssim \prod_{m=1}^k \|f_m\|_{L^p(\d\mu_m)}
    \end{equation*}
    holds for some $p\in[1,\infty]$, then
    \begin{equation*}
        q\geq \frac{2d}{\sum_{m=1}^k \ubd\supp(\mu_m)}.
    \end{equation*}
\end{proposition}
\begin{proof}
    Let $d>s>D_2(\mu_1*\cdots*\mu_k)$, H\"older's inequality with exponents $\frac{q}{2}$ and $\frac{q}{q-2}$ gives
    \begin{equation*}
        \infty = I_s(\mu_1*\cdots*\mu_k) \approx \int_{\rd} \big|\widehat{\d\mu_1}\cdots\widehat{\d\mu_k}(\xi)\big|^2 |\xi|^{s-d}\,\d\xi \leq \big\|\widehat{\d\mu_1}\cdots\widehat{\d\mu_k} \big\|_{L^q(\rd)}^2 \left( \int_{|\xi|\geq1} |\xi|^{\frac{q(s-d)}{q-2}}\,\d\xi \right)^{\frac{q-2}{q}},
    \end{equation*}
    where the last integral is finite if $q<\frac{2d}{s}$. Therefore, if a multilinear Fourier extension estimate holds then we must have $q\geq \frac{2d}{D_2(\mu_1*\cdots*\mu_k)}$. Since the $L^q$-dimensions are non-increasing in $q$
    \begin{equation*}
        D_2(\mu_1*\cdots*\mu_k) \leq D_0(\mu_1*\cdots*\mu_k) = \lbd \supp (\mu_1*\cdots*\mu_k) \leq \sum_{m=1}^k \ubd\supp(\mu_m).
    \end{equation*}
    Therefore, multilinear Fourier extension estimates hold only if
    \begin{equation*}
        q\geq \frac{2d}{D_2(\mu_1*\cdots*\mu_k)} \geq \frac{2d}{\sum_{m=1}^k \ubd\supp(\mu_m)},
    \end{equation*}
    as we wanted to prove.
\end{proof}
This condition (see Figure \ref{figure:singularrange}) can be thought of as the analogue of the condition $q>\frac{2d}{d-1}$ in Conjecture~\ref{generalklinearapr422}. Note that this `top lid' condition in the smooth setting does not include the endpoint, whereas Proposition~\ref{prop:topLid} does. In the smooth setting, the necessary condition follows from stronger assumptions on the measures that we do not have here (see \cite[Section~2.7]{TVV1}).

\begin{remark}\label{Trainormult-rem-220126} In Proposition 5.3 of \cite{Trainor}, Trainor gives the following necessary condition for a multilinear Fourier extension estimate to hold: let $\mu_{1},\ldots,\mu_{k}$ be compactly supported measures on $\mathbb{R}^{d}$ for which there exist sequences of points $\{x_{m,j}\}_{j=1}^{\infty}$ for each $m=1,\ldots,k$ and of radii $\{r_{j}\}_{j=1}^{\infty}$ such that $r_{j}\searrow 0$ and
\begin{equation}\label{Trainor-cond-V2-270126}
    \mu_{m}(B(x_{m,j},r_{j}))\approx r_{j}^{\gamma_{m}}.
\end{equation}
Under these conditions, if $\mathcal{R}^{\ast}_{\mu_{1},\ldots,\mu_{k}}(p\times\cdots\times p\rightarrow q)$ holds, then 
\begin{equation}\label{Trainor-cond-V2}
    q\geq \frac{dp'}{\left(\sum_{m=1}^{k}\gamma_{m}\right)}.
\end{equation}
\end{remark}

In order to compare the ranges from Theorem \ref{nec-sing} and Remark \ref{Trainormult-rem-220126} for $d=1$, let us establish a framework under which both results fit. Let $\mu_{1},\ldots,\mu_{k}$ be $k$ compactly supported measures in $\mathbb{R}$ such that
\begin{enumerate}[label=\textbf{(H\arabic*)}]
  \item \label{hyp:H1} $\mu_{m}$ satisfies the \textit{ball condition}: for $0\leq \alpha_{m} \leq 1$, $1\leq m\leq k$, given $\varepsilon>0$, there is $r_{\varepsilon}<1/4$ such that
  \begin{equation*}
        r^{\alpha_m+\varepsilon}\lesssim \mu_{m}\big(B(x,r)\big) \lesssim r^{\alpha_m},
\end{equation*}
for $x\in \textnormal{supp}{(\mu_{m})}$ and $0<r< r_{\varepsilon}$.
  \item \label{hyp:H2} $\mu_{m}$ has \textit{Fourier decay}: for $0\leq\beta_{m}\leq 1$, $1\leq m\leq k$,
  \begin{equation*}
        |\widehat{\mathrm{d}\mu_m}(\xi)|\lesssim |\xi|^{-\beta_m/2}.
    \end{equation*}
\end{enumerate}

Theorem~\ref{nec-sing} provides an example of $k$ measures $\mu_{1},\ldots,\mu_{k}$ satisfying \ref{hyp:H1} and \ref{hyp:H2}. On the other hand, any measure $\mu_{m}$ satisfying \ref{hyp:H1} also satisfies the hypothesis of \cite[Proposition~5.3]{Trainor} for $\gamma_{m}=\alpha_{m}$. A direct computation shows that, under \ref{hyp:H1} and \ref{hyp:H2}, Theorem~\ref{nec-sing} imposes more restrictive necessary conditions for $\mathcal{R}^{\ast}_{\mu_{1},\ldots,\mu_{k}}(p\times\cdots\times p\rightarrow q)$ to hold than \cite[Proposition~5.3]{Trainor} in the regime $\sum_{m=1}^k \alpha_m <1$ (see Figure \ref{figure:singularrange}). Note that \ref{hyp:H1} gives $\hd \textnormal{supp}{(\mu_{m})} = \pd \textnormal{supp}{(\mu_{m})} =\ubd\supp(\mu_m) = \alpha_{m}$ (see the proof of Theorem 5.7 of \cite{Mat12}), thus $\mu_{1}*\cdots*\mu_{k}$ is singular if $\sum_{m=1}^k \alpha_{m}<1$. On the other hand, in the setting of the upcoming Proposition~\ref{thm:mainthm} (where we need this convolution to be absolutely continuous) the necessary condition obtained in \cite{Trainor} will be more restrictive. We stress that \eqref{cond2-280126} is a technical assumption of the particular example exhibited in Theorem~\ref{nec-sing}, which is why it is not included in the framework above.

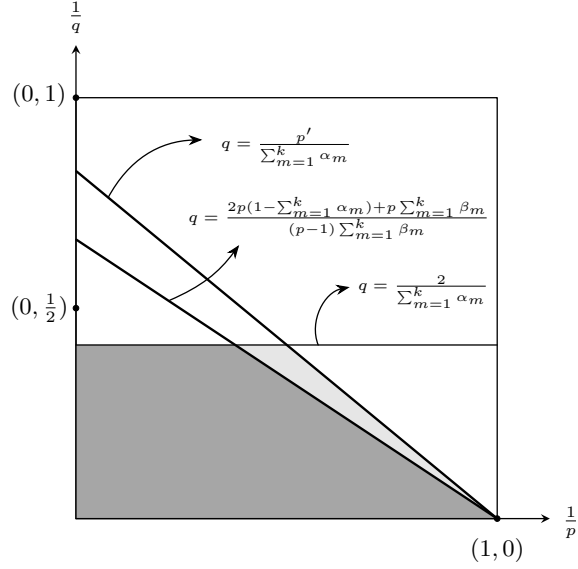
\begin{figure}[ht]
\centering
\begin{tikzpicture}[line cap=round,line join=round,>=Stealth,x=1cm,y=1cm, decoration={brace,amplitude=4pt}, scale=0.7]

\clip(-2.5,-1) rectangle (8.5,10);

%%%
%\filldraw[line width=0.5pt,color=black,fill=black,fill opacity=0.0] (-11,0) -- (-2,0) -- (-2,9) -- (-11,9) -- cycle;

\filldraw[line width=0.5pt,color=black,fill=black,fill opacity=0.0] (-1,0) -- (7,0) -- (7,8) -- (-1,8) -- cycle;

\fill[gray!20] (7,0) -- (3,3.3) -- (-1,3.3) -- cycle;
\fill[gray!70] (-1,0) -- (6.97,0) -- (2,3.3) -- (-1,3.3) -- cycle;

%%%

\filldraw[line width=0.5pt,color=black,fill=black,fill opacity=0.0] (10,0) -- (18,0) -- (18,8) -- (10,8) -- cycle;

\fill[gray!20] (18,0) -- (13.6,4) -- (10,4) -- cycle;
\fill[gray!70] (10,0) -- (17.97,0) -- (12.62,4) -- (10,4) -- cycle;

%%%

\begin{scope}
    \clip (-1.5,-1.5) rectangle (9,9.5);
    \draw [line width=0.5pt, -stealth] (-1,0)-- (8,0);
    \draw [line width=0.5pt, -stealth] (-1,0)-- (-1,9);
    \draw [line width=0.9pt] (7,0)-- (-.99,5.3);
    \draw [line width=0.9pt] (7,0)-- (-.99,6.6);
\end{scope}

\begin{scope}
    \clip (9.5,-1.5) rectangle (20,9.5);
    \draw [line width=0.5pt, -stealth] (10,0)-- (19,0);
    \draw [line width=0.5pt, -stealth] (10,0)-- (10,9);
    \draw [line width=0.9pt] (18,0)-- (10,6);
    \draw [line width=0.9pt] (18,0)-- (10,7.3);
\end{scope}

%%%

\draw [anchor=north] (-1.7,8.45) node {\small $(0,1)$};
\filldraw (-1,8) circle (1.5pt);
\draw [anchor=north] (-1.7,4.45) node {\small $(0,\tfrac{1}{2})$};
\filldraw (-1,4) circle (1.5pt);
\draw [anchor=north] (9.3,8.45) node {\small $(0,1)$};
\filldraw (10,8) circle (1.5pt);
\draw [anchor=north] (9.3,4.45) node {\small $(0,\tfrac{1}{2})$};
\filldraw (10,4) circle (1.5pt);
\draw [anchor=north] (7,-0.2) node {\small $(1,0)$};
\filldraw (7,0) circle (1.5pt);
%\draw [anchor=north] (2.9,-.4) node {\tiny $\boxed{\sum_{m=1}^{k}\alpha_{m}<1\textnormal{ and }\sum_{m=1}^{k}\beta_{m}<1}$};
\draw [anchor=north] (18,-0.2) node {\small $(1,0)$};
\filldraw (18,0) circle (1.5pt);
%\draw [anchor=north] (13.9,-.4) node {\tiny $\boxed{\sum_{m=1}^{k}\alpha_{m}<1\textnormal{ and }\sum_{m=1}^{k}\beta_{m}\geq 1}$};
\draw [anchor=north] (-1,10) node {\small $\frac{1}{q}$};
\draw [anchor=north] (10,10) node {\small $\frac{1}{q}$};
\draw [anchor=north] (8.4,0.4) node {\small $\frac{1}{p}$};
\draw [anchor=north] (19.4,0.4) node {\small $\frac{1}{p}$};

\draw [->, line width=0.5pt]
  (-.39,6.1) to[bend left=30] (1.4,7.2);
\draw [anchor=north] (3,7.6) node {\tiny $q=\frac{p'}{\sum_{m=1}^{k}\alpha_{m}}$};

\draw [->, line width=0.5pt]
  (10.61,6.8) to[bend left=30] (12.4,7.2);
\draw [anchor=north] (14,7.6) node {\tiny $q=\frac{p'}{(\sum_{m=1}^{k}\alpha_{m})}$};

\draw [->, line width=0.5pt]
  (.78,4.15) to[bend right=15] (2.1,5.2);
\draw [anchor=north] (4,6.3) node {\tiny $q=\frac{2p(1-\sum_{m=1}^k\alpha_m) + p\sum_{m=1}^k\beta_m}{(p-1)\sum_{m=1}^{k}\beta_m}$};

\draw [->, line width=0.5pt]
  (12.47,4.15) to[bend right=15] (13.8,5.2);
\draw [anchor=north] (15,6.3) node {\tiny $q=\frac{2p(1-\sum_{m=1}^k\alpha_m) + p\sum_{m=1}^k\beta_m}{(p-1)\sum_{m=1}^{k}\beta_m}$};

\draw [anchor=north] (5.6,4.8) node {\tiny $q=\frac{2}{\sum_{m=1}^{k}\alpha_{m}}$};

\draw [->, line width=0.5pt]
  (3.6,3.3) to[bend left=50] (4.2,4.4);

\draw [anchor=north] (16.6,4.8) node {\tiny $q=2$};

\draw [->, line width=0.5pt]
  (14.8,4) to[bend left=50] (16,4.6);

\draw [line width=0.5pt] (-1,3.3)-- (7,3.3);

\draw [dashed, line width=0.5pt] (10,4)-- (18,4);

\end{tikzpicture}
\caption{If $\sum_{m=1}^{k}\alpha_{m}<1$, the dark grey region represents the necessary conditions of Theorem \ref{nec-sing} intersected with $q\geq\frac{2}{\sum_{m=1}^{k}\alpha_{m}}$ from Proposition~\ref{prop:topLid}, since \ref{hyp:H1} implies $\ubd\supp(\mu_m) = \alpha_m$. Notice that the latter are more restrictive than the conditions from Proposition 5.3 of \cite{Trainor} (represented by the union of the dark and light grey regions) in a setting with $k$ measures $\mu_{1},\ldots,\mu_{k}$ satisfying \ref{hyp:H1} and \ref{hyp:H2}.} \label{figure:singularrange}
\end{figure}

A simple computation shows that, since $\beta_{m}\leq 2\alpha_{m}$ must hold (see Section 3 of \cite{Mit02}), the necessary condition
\begin{equation*}
    q \geq \frac{2p(1-\sum_{m=1}^k\alpha_m) + p\sum_{m=1}^k\beta_m}{(p-1)\sum_{m=1}^{k}\beta_m}
\end{equation*}
covers more pairs $(p,q)$ than those following from the linear theory. To illustrate this, let $k=2$, $p=2$ and $0<\beta_{2}=\alpha_{2}\leq\alpha_{1}=\beta_{1}\leq 1$. Theorem \ref{nec-sing} gives
    \begin{equation*}
        q\geq \frac{4}{\alpha_{1}+\alpha_{2}}-2.
    \end{equation*}
On the other hand, Theorem \ref{thm:ST} implies
\begin{equation*}
    \left\|\widehat{f_{1}\mathrm{d}\mu_{1}}\widehat{f_{2}\mathrm{d}\mu_{2}}\right\|_{L^q(\mathbb{R})}\leq \|\widehat{f_{1}\mathrm{d}\mu_{1}}\|_{L^{2q}(\mathbb{R})}\|\widehat{f_{2}\mathrm{d}\mu_{2}}\|_{L^{2q}(\mathbb{R})} \lesssim\|f_{1}\|_{L^2(\d\mu_1)}\|f_{2}\|_{L^2(\d\mu_2)},
\end{equation*}
for $q\geq\frac{1}{2}\max_{i} \left(\frac{4}{\alpha_{i}}-2\right)=\frac{2}{\alpha_{2}}-1$, hence the range
\begin{equation*}
   \frac{2}{\alpha_{2}}-1>q\geq\frac{4}{\alpha_{1}+\alpha_{2}}-2
\end{equation*}
is a genuinely `multilinear gap' that is not accessible through separate linear estimates. This suggests that there are manifestations of transversality beyond the absolute continuity of the convolution of the underlying measures (see subsection~\ref{sec:LI}).

\noindent\textbf{A sufficient condition.} The following proposition gives a condition on the convolution $\mu_{1}\ast\cdots\ast\mu_{k}$ of $k$ measures so that $\R^*_{\mu_1,\ldots,\mu_k}(p\times\cdots\times p\to q)$ holds. The proof of this particular result follows the lines of \cite{Chen,Trainor}, recovers it when setting $p = q'$, and does not involve a substantially new idea. We believe, however, that its applicability makes it valuable; as we will see, it is a great source of examples.
\begin{proposition}\label{thm:mainthm}
    Let $\mu_1,\ldots,\mu_k$ be finite, compactly supported, Borel measures on $\rd$. If for some $p>1$ and $q\geq\max(2,p')$, $\mu_1*\cdots*\mu_k\in L^{\frac{q(p-1)}{q(p-1)-p}}(\rd)$, then
    \begin{equation}\label{ineq1-230726}
        \Big\| \widehat{f_1\d\mu_1}\cdots\widehat{f_k\d\mu_k} \Big\|_{L^q(\rd)} \lesssim \prod_{i=1}^k \|f_i\|_{L^p(\d\mu_i)}
    \end{equation}
    holds for all $f_i\in L^p(\d\mu_i)$. If $p=1$, then \eqref{ineq1-230726} holds trivially with $q=\infty$ without any assumption on $\mu_1*\cdots*\mu_k$.
\end{proposition}

\begin{remark} Alternatively, one can rewrite the previous proposition as follows. Let $\mu_1,\ldots,\mu_k$ be finite, compactly supported, Borel measures on $\mathbb{R}^d$. If for some $p>1$, $1<q\leq\infty$ and $p'q'\geq 2$ it holds that $\mu_1*\cdots*\mu_k\in L^q(\mathbb{R}^d)$, then
    \begin{equation*}
        \Big\| \widehat{f_1\d\mu_1}\cdots\widehat{f_k\d\mu_k} \Big\|_{L^{p'q'}(\rd)} \lesssim \prod_{i=1}^k \|f_i\|_{L^p(\d\mu_i)}
    \end{equation*}
    holds for all $f_i\in L^p(\d\mu_i)$. 
    
\end{remark}

Determining the conditions that guarantee that the convolution of measures will be absolutely continuous is a fundamental question in fractal geometry, and the different notions of dimensions defined in Section~\ref{sec:prelim} lie at its core. In \cite{SS} the authors gave an answer to this question using the $L^q$ and Fourier dimensions of the measures involved. The following two corollaries follow from applying \cite[Theorem~4.4]{SS} and \cite[Theorem~D]{SS} respectively, to Proposition~\ref{thm:mainthm}.
\begin{corollary}\label{cor1} Let $\mu,\nu$ be two finite, compactly supported, Borel measures on $\rd$. Let $p\geq1$, $q\geq 2p'$, and $p_0 := \frac{q(p-1)}{q(p-1)-p}$. If
    \begin{equation*}
            d-D_{p_0}(\mu) < \fd\nu.
    \end{equation*}
    Then
    \begin{equation*}
        \big\| \widehat{f\d\mu}\,\widehat{g\d\nu} \big\|_{L^q(\rd)} \lesssim \|f\|_{L^p(\d\mu)} \|g\|_{L^p(\d\nu)}
    \end{equation*}
    holds for all $f\in L^{p}(\d\mu)$, $g\in L^{p}(\d\nu)$.
\end{corollary}

\begin{corollary}\label{cor2} For $j=1,2$, let $\boldsymbol{\varphi_j} = (\lambda_j x + a_{ji})_{i=1}^{m_j}$ be an IFS satisfying the strong separation condition, with $\frac{\log|\lambda_2|}{\log|\lambda_1|}\notin\mathbb{Q}$ and $\frac{\log m_1}{|\log\lambda_1|} + \frac{\log m_2}{|\log\lambda_2|} >1$. Let $\mu_j$ be a self-similar measure on $\boldsymbol{\varphi_j}$. For $p>1$, $q\geq 2p'$, let 
\[
    p_0=\frac{q(p-1)}{q(p-1)-p},
\]
which satisfies $p_{0}\leq 2$ by hypothesis. If $D_{p_0}(\mu_1) + D_{p_0}(\mu_2)>1$, then there exists a set $E\subseteq\mathbb{R}$ of Hausdorff dimension zero such that for all $u\in\mathbb{R}\backslash E$, the estimate
    \begin{equation}\label{eq:cor2}
        \|\widehat{f\d\mu_1}\widehat{g\d T_u\mu_2}\|_{L^q(\mathbb{R})} \lesssim \|f\|_{L^p(\d\mu_1)}\|g\|_{L^{p}(\d T_u\mu_2)}
    \end{equation}
    holds for all $f\in L^{p}(\d\mu_1)$, $g\in L^{p}(\d T_u\mu_2)$, where $T_u(x) = ux$.
\end{corollary}
Using the fact that $\fd (\mu_1*\mu_2) \geq \min\{\fd\mu_1 + \fd\mu_2, d\}$ one can trivially generalise Corollary~\ref{cor1} to obtain $k$-linear estimates whenever $d-D_{p_0}(\mu_1) < \sum_{m=2}^k \fd\mu_m$, but we do not expect this to yield the best possible bounds; to illustrate this, suppose we are interested in proving trilinear Fourier extension estimates. Write
\begin{equation*}
    |\widehat{f_{1}\mathrm{d}\mu_{1}} \widehat{f_{2}\mathrm{d}\mu_{2}}\widehat{f_{3}\mathrm{d}\mu_{3}}|=|\widehat{f_{1}\mathrm{d}\mu_{1}}\widehat{f_{2}\mathrm{d}\mu_{2}}|^{\frac{1}{2}}|\widehat{f_{2}\mathrm{d}\mu_{2}}\widehat{f_{3}\mathrm{d}\mu_{3}}|^{\frac{1}{2}}|\widehat{f_{3}\mathrm{d}\mu_{3}}\widehat{f_{1}\mathrm{d}\mu_{1}}|^{\frac{1}{2}}
\end{equation*}
and apply H\"older's inequality to bound the right-hand side by norms of bilinear operators for which either Corollary \ref{cor1} or \ref{cor2} can be applied. Even though we would be getting non-trivial trilinear estimates, applying H\"older destroys the simultaneous interaction between the three measures. In the smooth setting, for instance, the trilinear estimates that follow from the argument above are worse than those obtained in \cite{BCT}. It would be interesting to obtain non-trivial extensions of the work in \cite{SS} to multiple convolutions, not only to provide more examples for Proposition~\ref{thm:mainthm}, but also to shed light on new manifestations of `fractal transversality' regarding multiple measures simultaneously.

The following shows that Proposition~\ref{thm:mainthm} gives multilinear estimates for measures for which \cite[Theorem~5.1]{Trainor} did not apply.
\begin{example}\label{ex:LpLinfty}
    Let $0<\alpha,\beta<1$, define the measures $\d\mu= x^{-\alpha}1_{[0,1]} \d x$, $\d\nu= x^{-\beta}1_{[0,1]} \d x$. Their convolution is
    \begin{equation*}
        \mu\ast\nu (x) = \int_{\max\{0,x-1\}}^{\min\{1,x\}} (x-y)^{-\alpha} y^{-\beta} \,\d y= \underbrace{\mathbbm{1}_{(0,1)}(x)\int_{0}^{x} (x-y)^{-\alpha} y^{-\beta} \,\d y}_{M_{1}(x)} + \underbrace{\mathbbm{1}_{[1,2)}(x)\int_{x-1}^{1} (x-y)^{-\alpha} y^{-\beta} \,\d y}_{M_{2}(x)}.
    \end{equation*}
    The change of variables $y=tx$, $0\leq t\leq 1$, gives $M_{1}(x)=x^{1-\alpha-\beta}B(1-\alpha,1-\beta)$, where $B(z_1,z_2)$ is the Beta function, which converges for $z_1,z_2 > 0$. The same change of variables turns $M_{2}(x)$ into a term uniformly bounded in $x$, hence only $M_{1}(x)$ determines $L^{p}$ integrability of $\mu\ast\nu$. Observe that $\mu\ast\nu\in L^{p}$ for every $1\leq p< 2$ if $\alpha=\beta=\frac{3}{4}$, but clearly $\mu*\nu\notin L^\infty(\mathbb{R})$.
    \end{example}

Corollary \ref{cor2} is also a good source of examples, although less explicit than those following from Corollary \ref{cor1}.

\begin{example}\label{example-010226}
    Consider the two IFSs
    \begin{align*}
        \boldsymbol{\varphi_1} &= \big\{ \tfrac{1}{4} x + a_1, \tfrac{1}{4}x + a_2, \tfrac{1}{4}x + a_3 \big\};\\
        \boldsymbol{\varphi_2} &= \big\{ \tfrac{1}{3} x + b_1, \tfrac{1}{3}x + b_2\big\},
    \end{align*}
    where $a_i,b_i\in[0,1]$ can be chosen arbitrarily so that both IFSs satisfy the strong separation condition. Let $\mu_1,\mu_2$ be the self-similar measures on $\boldsymbol{\varphi_1},\boldsymbol{\varphi_2}$ with probabilities $(\rho_1,\rho_2,\rho_3)$ and $(\gamma,1-\gamma)$ respectively. Note that $\frac{\log 3}{\log 4} + \frac{\log 2}{\log 3} \approx 1.42\ldots >1$.

    Let $p_0$ be such that
    \begin{equation}\label{eq:conditionP0}
        \frac{\log (\rho_1^{p_0} + \rho_2^{p_0} + \rho_3^{p_0})}{(1-p_0)\log4} + \frac{\log (\gamma^{p_0} + (1-\gamma)^{p_0})}{(1-p_0)\log3} >1.
    \end{equation}
    Since the $L^q$ dimensions are non-increasing in $q$, there will be a unique value $\widetilde{p_0}=\widetilde{p_0}(\rho_{1},\rho_{2},\rho_{3},\gamma)$ such that \eqref{eq:conditionP0} holds for all $p_0\in[1,\widetilde{p_0})$. This will imply that \eqref{eq:cor2} holds for any $p\geq1$ and $q\geq 2p'$ satisfying $1\leq\frac{q(p-1)}{q(p-1)-p}\leq \widetilde{p_0}$.
    
    As an example, let $\rho_1 = 0.1$, $\rho_2 = 0.65$. For $p_0=2$ and any $\gamma\in[0.3,0.7]$, there exists $E\subseteq\mathbb{R}$ with $\hd E = 0$ such that for any $u\in \mathbb{R}\backslash E$,
    \begin{equation*}
        \|\widehat{f\d\mu_1}\widehat{g\d \nu}\|_{L^q(\mathbb{R})} \lesssim \|f\|_{L^p(\d\mu_1)}\|g\|_{L^{p}(\d \nu)}
    \end{equation*}
    holds for any $q\geq 2p'$, $p\geq1$, where $\nu$ is the self-similar measure $T_u\mu_{2}$ on the IFS $\big\{ \tfrac{x}{3} + ub_1, \tfrac{x}{3} + ub_2\big\}$ with probabilities $(\gamma,1-\gamma)$.
\end{example}

The following examples use linear theory to obtain bilinear estimates for measures with singular convolutions for which we cannot apply the upcoming Proposition~\ref{thm:mainthm}. The first one is genuinely fractal, whereas the second uses classical Fourier restriction theory for curves and surfaces in $\mathbb{R}^{3}$. The latter shows that Proposition~\ref{thm:mainthm} is still far from providing a satisfactory answer to the main question raised in the introduction.
\begin{example}\label{ex:estimateSingular} Let $0<\beta_1\leq\beta_2<d$ such that $\beta_1 + \beta_2 <d$. Let $E_{1}$ be a Salem set of dimension $\beta_1$, i.e. $\fd E_1 = \hd E_1 = \beta_1$, and $E_2$ be such that $\hd E_2= \pd E_2 = \beta_2/2$. Define $F_{2}$ as the image of $E_{2}$ under Brownian motion. The set $F_{2}$ satisfies $\fd F_{2} = \hd F_{2} = \pd F_{2} = \beta_2$. This is to guarantee that
\begin{equation*}
    \hd(E_{1} + F_{2}) \leq \hd(E_{1} \times F_{2}) \leq \hd E_{1} + \pd F_{2} = \beta_1 + \beta_2 <d.
\end{equation*}
This way, any measure supported on $E_{1} + F_{2}$ is singular. Let $\mu_{1},\mu_{2}$ be measures on $E_{1}$ and $F_{2}$ respectively such that for some $\varepsilon>0$,
\begin{equation*}
    \big| \widehat{\d\mu_{1}}(\xi) \big|^{2} \lesssim |\xi|^{-(\beta_1- \varepsilon)};\qquad  \big| \widehat{\d\mu_{2}}(\xi) \big|^{2} \lesssim |\xi|^{-(\beta_2- \varepsilon)}.
\end{equation*}
Since $\mu_1*\mu_{2}$ is a measure supported on $E_{1} + F_{2}$, then $\mu_1 *\mu_2$ is singular. Also, recall that the $L^\infty$-dimension is bounded below by half the Fourier dimension. Therefore, by Cauchy--Schwarz and the Stein--Tomas estimate, for any $q>\frac{2d}{\beta_1-\varepsilon}$,
\begin{equation*}
    \|\widehat{f\d\mu_1} \widehat{g\d\mu_2}\|_{L^q(\rd)} \leq \|\widehat{f\d\mu_1}\|_{L^{2q}(\rd)} \|\widehat{g\d\mu_2}\|_{L^{2q}(\rd)} \lesssim \|f\|_{L^2(\d\mu_1)} \|g\|_{L^2(\d\mu_2)}.
\end{equation*}

The example above shows that there are nontrivial multilinear Fourier extension estimates even if the convolution of the underlying measures is singular. For the interested reader, this naturally raises the question of whether there is a suitable notion of `transversality' between these measures whose convolution is singular (i.e. a way of exploiting their interaction) that allows one to extend the range implied by the linear theory. The next example achieves that indirectly by relying on higher-dimensional phenomena as a substitute for transversality.
\end{example}

\begin{example}\label{example1-22072026} For $\varepsilon>0$, let $\gamma_{j}:I_{\varepsilon}=(-\varepsilon,\varepsilon)\rightarrow\mathbb{R}^{3}$, $1\leq j\leq 2$, be curve parametrisations given by 
\begin{equation*}
    \eqalign{
    \gamma_{1}(s)&\displaystyle=(s,s^{2},s^{3}), \cr
    \gamma_{2}(t)&\displaystyle=(t^{3},t^{2},t).
    }
\end{equation*}
Let $\chi:(-\varepsilon,\varepsilon)\rightarrow\mathbb{R}$ be a smooth bump and define $\mu_{j}:=(\gamma_{j})_{\ast}(\chi(x)\mathrm{d}x)$, $1\leq j\leq 2$. By Drury's restriction theorem (see \cite{Drury}) for the moment curve in $\mathbb{R}^{3}$ (applied in its dual formulation) and Cauchy--Schwarz,
\begin{equation}\label{eq1-sing-example-220726}
    \|\widehat{f_{1}\mathrm{d}\mu_{1}}\widehat{f_{2}\mathrm{d}\mu_{2}}\|_{L^{6}(\mathbb{R}^{3})}\leq \|\widehat{f_{1}\mathrm{d}\mu_{1}}\|_{L^{12}(\mathbb{R}^{3})}\|\widehat{f_{2}\mathrm{d}\mu_{2}}\|_{L^{12}(\mathbb{R}^{3})}\lesssim \|f_{1}\|_{L^{2}(\mathrm{d}\mu_{1})}\|f_{2}\|_{L^{2}(\mathrm{d}\mu_{2})}
\end{equation}
On the other hand, the convolution $\mu_{1}\ast\mu_{2}$ is singular because it is supported in the smooth manifold $\mathcal{M}$ parametrised by

\begin{equation*}
    \Phi(s,t)=(s+t^{3},s^{2}+t^{2},s^{3}+t),\quad -\varepsilon\leq s,t\leq \varepsilon.
\end{equation*}
One can check that the Gaussian curvature of $\mathcal{M}$ at $\Phi(0,0)$ is strictly positive, hence by choosing $\varepsilon>0$ small enough we can guarantee that $\Phi(I_{\varepsilon}\times I_{\varepsilon})\subset\mathbb{R}^{3}$ is a two-dimensional surface with non-vanishing Gaussian curvature. For $H(\Phi(s,t))=f_{1}(\gamma_{1}(s))f_{2}(\gamma_{2}(t))$, the Stein--Tomas theorem applied to the measure $\mu=\mu_{1}\ast\mu_{2}=\Phi_{\ast}(\chi(s)\chi(t)\mathrm{d}s\,\mathrm{d}t)$ gives

\begin{equation}\label{eq2-sing-example-220726}
 \eqalign{
 \displaystyle\|\widehat{f_{1}\mathrm{d}\mu_{1}}\widehat{f_{2}\mathrm{d}\mu_{2}}\|_{L^{4}(\mathbb{R}^{3})} &\displaystyle=\|\widehat{H\mathrm{d}\mu}\|_{L^{4}(\mathbb{R}^{3})} \cr
 &\displaystyle\lesssim \|H\|_{L^{2}(\mathrm{d}\mu)} \cr
 &\displaystyle=\left(\int_{I_{\varepsilon}}\int_{I_{\varepsilon}}f_{1}(s)f_{2}(t)\chi(s)\chi(t)\mathrm{d}s\mathrm{d}t\right)^{\frac{1}{2}} \cr
 &\displaystyle=\|f_{1}\|_{L^{2}(\mathrm{d}\mu_{1})}\|f_{2}\|_{L^{2}(\mathrm{d}\mu_{2})}.
 }
\end{equation}

This example shows that, even though $\mu_{1}\ast\mu_{2}$ is singular, there are bilinear restriction estimates for the operator $(f_{1},f_{2})\mapsto \widehat{f_{1}\mathrm{d}\mu_{1}}\widehat{f_{2}\mathrm{d}\mu_{2}}$ that are not obtainable by separating the oscillatory factors $\widehat{f_{1}\mathrm{d}\mu_{1}}$ and $\widehat{f_{2}\mathrm{d}\mu_{2}}$ with Cauchy-Schwarz.
\end{example}

\section{Proof of Proposition~\ref{thm:mainthm}}\label{proof-main-abs-310126}

Let $\varphi\in\mathcal{S}_+(\rd)$ be a function supported in $B(0,1)$ with $\|\varphi\|_{L^1(\rd)}=1$, and define an approximate identity $(\varphi_\varepsilon)_{\varepsilon>0}$ by $\varphi_\varepsilon(x) = \varepsilon^{-d}\varphi(x/\varepsilon)$. For each $m=1,\ldots,k$, let $\mu_{m,\varepsilon} = \varphi_\varepsilon * \mu_m$, which converges weakly to $\mu_{m}$ as $\varepsilon\to0$.   
    
    By Hausdorff--Young's inequality, since $q\geq2$
    \begin{equation*}
         \Big\| \widehat{f_1\mu_{1,\varepsilon}}\cdots\widehat{f_k\mu_{k,\varepsilon}} \Big\|_{L^q(\rd)} \leq  \Big\| f_1\mu_{1,\varepsilon}*\cdots*f_k\mu_{k,\varepsilon} \Big\|_{L ^{q'}(\rd)}.
    \end{equation*}
    Write
    \begin{equation*}
        F(\xi,\eta) = f_1(\eta_1)\prod_{j=2}^k f_j(\eta_j - \eta_{j-1});\qquad M_\varepsilon= \mu_{1,\varepsilon}(\eta_1)\prod_{j=2}^k \mu_{j,\varepsilon}(\eta_j-\eta_{j-1}),
    \end{equation*}
    where $\eta=(\eta_{1},\ldots,\eta_{k-1})\in\mathbb{R}^{k-1}$ and $\xi=\eta_{k}$. Then by the definition of convolution and H\"older's inequality for some $p\geq1$,
    \begin{align*}
        \Big\| f_1\mu_{1,\varepsilon}*\cdots*f_k\mu_{k,\varepsilon} \Big\|_{L ^{q'}(\rd)} &= \left( \int_{\rd} \left| \int_{\mathbb{R}^{(k-1)d}} M_\varepsilon(\xi,\eta)^{\frac{1}{p}}F(\xi,\eta) M_\varepsilon(\xi,\eta)^{\frac{1}{p'}} \,\d\eta \right|^{q'} \,\d\xi \right)^{\frac{1}{q'}}\\
        &\leq \left(\int_{\mathbb{R}^{d}}\left|\int_{\mathbb{R}^{(k-1)d}}M_{\varepsilon}(\xi,\eta) \,\mathrm{d}\eta\right|^{\frac{q'}{p'}}\cdot\left|\int_{\mathbb{R}^{(k-1)d}}|F(\xi,\eta)|^{p}M_{\varepsilon}(\xi,\eta) \,\d\eta\right|^{\frac{q'}{p}}\mathrm{d}\xi\right)^{\frac{1}{q'}} \\
        &=\Bigg\| \bigg( \int_{\mathbb{R}^{(k-1)d}} M_\varepsilon(\cdot,\eta) \,\d\eta \bigg)^{\frac{1}{p'}}  \bigg( \int_{\mathbb{R}^{(k-1)d}} |F(\cdot,\eta)|^p M_\varepsilon(\cdot,\eta) \,\d\eta \bigg)^{\frac{1}{p}} \Bigg\|_{L^{q'}(\rd)}.
    \end{align*}
    By H\"older's inequality, with $r = \frac{pq}{q(p-1)-p}$, $\frac{1}{q'} = \frac{1}{p} + \frac{1}{r}$
    \begin{align*}
        \Bigg\| \bigg( \int_{\mathbb{R}^{(k-1)d}} &M_\varepsilon(\cdot,\eta) \,\d\eta \bigg)^{\frac{1}{p'}}  \bigg( \int_{\mathbb{R}^{(k-1)d}} |F(\cdot,\eta)|^p M_\varepsilon(\cdot,\eta) \,\d\eta \bigg)^{\frac{1}{p}} \Bigg\|_{L^{q'}(\rd)} \\
        &\leq \Bigg\| \bigg( \int_{\mathbb{R}^{(k-1)d}} M_\varepsilon(\cdot,\eta) \,\d\eta \bigg)^{\frac{1}{p'}} \Bigg\|_{L^r(\rd)} \Bigg\|  \bigg( \int_{\mathbb{R}^{(k-1)d}} |F(\cdot,\eta)|^p M_\varepsilon(\cdot,\eta) \,\d\eta \bigg)^{\frac{1}{p}}  \Big\|_{L^p(\rd)}
    \end{align*}
    For the first term note that
    \begin{align*}
        \Bigg\| \bigg( \int_{\mathbb{R}^{(k-1)d}} M_\varepsilon(\cdot,\eta) \,\d\eta \bigg)^{\frac{1}{p'}} \Bigg\|_{L^r(\rd)}  &= \Bigg\|  \int_{\mathbb{R}^{(k-1)d}} M_\varepsilon(\cdot,\eta) \,\d\eta  \Bigg\|_{L^{\frac{r}{p'}}(\rd)}^{\frac{1}{p'}} \\
        &=  \|  \mu_{1,\varepsilon}*\cdots*\mu_{k,\varepsilon} \|_{L^{\frac{r}{p'}}(\rd)}^{\frac{1}{p'}} \\
        &= \|  \mu_{1,\varepsilon}*\cdots*\mu_{k,\varepsilon} \|_{L^{\frac{q(p-1)}{q(p-1)-p}}(\rd)}^{\frac{1}{p'}}.
    \end{align*}
    Finally, the second term is
    \begin{equation*}
        \Bigg\|  \bigg( \int_{\mathbb{R}^{(k-1)d}} |F(\cdot,\eta)|^p M_\varepsilon(\cdot,\eta) \,\d\eta \bigg)^{\frac{1}{p}}  \Big\|_{L^p(\rd)}= \prod_{i=1}^k \|f_i\|_{L^p(\mu_{i,\varepsilon})},
    \end{equation*}
    and letting $\varepsilon\to0$ yields the desired result.

\section{Proofs of the main results}\label{sec:MultiKnappV3}

Recall that the Knapp example for a smooth manifold consists of finding a function that gives a necessary condition for the linear estimate \eqref{rest1} by capturing flatness. The appropriate choice of such function is the characteristic function of a small cap. For simplicity, consider the sphere $(\mathbb{S}^{d-1},\d\sigma_{d-1})$ in $\mathbb{R}^{d}$, let $C_{\delta}$ be a $\delta^{\frac{1}{2}}\times\ldots\times\delta^{\frac{1}{2}}\times\delta$ cap of it centred at the north pole $e_{d}=(0,\ldots,0,1)$ and $g:=\chi_{C_{\delta}}$. One has
\begin{equation*}
    \frac{\|\widehat{g\d\sigma_{d-1}}\|_{L^q(\rd)}}{\|g\|_{L^2(\d\sigma_{d-1})}}\gtrsim\frac{\delta^{\frac{d-1}{2}-\frac{d+1}{2q}}}{\delta^{\frac{d-1}{4}}},
\end{equation*}
and letting $\delta\rightarrow 0$ implies $q\geq \frac{2(d+1)}{d-1}$ as a necessary condition for an $L^{2}$-based spherical estimate. In \cite{HL} and \cite{Che16}, the authors constructed Knapp-type example which proved the sharpness of the Stein--Tomas Theorem~\ref{thm:ST}.  The idea behind \cite{HL,Che16} is to randomise the construction of a Cantor-type set such that its natural measure has Fourier decay and satisfies a ball condition. However, to understand the `flatness' of the set, the construction simultaneously embeds in it an arithmetically structured Cantor set, small enough so that it does not affect these conditions. The functions that give the desired necessary condition are characteristic functions on neighbourhoods of this arithmetically structured set.

The constraints on $p$ and $q$ in the statement of the multilinear Conjecture \ref{generalklinearapr422} are more delicate: in \cite{TVV1}, Tao, Vargas and Vega identified two necessary conditions coming from considering \textit{squashed} and \textit{stretched} caps. These conditions predicted that one could prove multilinear estimates beyond the range of linear ones, which were obtained shortly after by Wolff \cite{Wolffcone} (in the $L^2$-based bilinear case for the cone) and Tao \cite{Tao1} (in the $L^{2}$-based bilinear case for the paraboloid).

The goal of this section is to provide a multilinear analogue of Chen's Knapp example from \cite{Che16}, which was based on Hambrook and \L aba's work \cite{HL}. Generally speaking, if we were to directly adapt the ideas from \cite{Che16}, there would be \textit{four} steps in the argument in which one passes from an identity to an inequality, which in turn implies a loss of information. The result presented here contains one extra technical feature compared to previous works: we require the arithmetic structures embedded in the support of each measure to be `linearly independent'. The latter condition will prevent us from losing information in one particular step of the argument that we see as a manifestation of `transversality', \textit{reducing the number of lossy steps from four to three}.

\subsection{Setup, arithmetic structure and linear independence}

Given $\phi:[2,\infty)\rightarrow(0,\infty)$ satisfying $\lim_{t\rightarrow\infty}\phi(t)=\infty$ such that $\phi(2t)\leq\phi(t)+C$ (for instance, $\phi(t)=(\log{t})^{\varepsilon}$ for any $\varepsilon>0$), define

\begin{equation*}
\psi(N)=\lceil\phi(2^N)^{1/2}\rceil+2,
\end{equation*}
and
\begin{equation*}
    \Psi(N)=\prod_{n=1}^{N}\psi(n).
\end{equation*}

For each $N\in\mathbb{N}$ and each $1\leq m\leq k$, given $0<\beta_{m}\leq\alpha_{m}<1$, we will construct three sequences related to an arithmetic progression $W_{N,m}$.
\begin{itemize}
    \item $\{d_{N,m}\}_{N\geq 1}$ will be the sequence of positive common differences between two consecutive elements of $W_{N,m}$.
    \item $\{\tau_{N,m}\}_{N\geq 1}$ will encode the number of elements of each $W_{N,m}$, i.e. $|W_{N,m}|=\tau_{N,m}$.
    \item $\{t_{N,m}\}_{N\geq 1}$ will be a sequence that registers the cardinality of the endpoints $A_{N,m}$ of the Cantor set at level $N$ which will contain the arithmetic progression $W_{N,m}$.
\end{itemize}

In the construction of the Cantor-type sets $E_{m}$ below, we will select a set $A_{N,m}\subseteq[\psi(N)]$ with $t_{N,m}$ elements that also contains $W_{N,m}$. Due to these constraints, the following conditions will have to be satisfied:
\begin{equation}\label{cond2-090126}
    1+(\tau_{N,m}-1)d_{N,m}\leq\psi(N).
\end{equation}
and
\begin{equation}\label{cond3-090126}
    \tau_{N,m}\leq t_{N,m}\leq\psi(N).
\end{equation}

As we will see, for a fixed $1\leq m\leq k$, at each stage $N\in\mathbb{N}$ of the construction of the Cantor set $E_{m}$ with associated measure $\mu_{m}$, the progression
$$W_{N,m}= \left\{1, 1 + d_{N,m}, \cdots, 1 + (\tau_{N,m}-1)d_{N,m}\right\}$$
will generate the \textbf{arithmetically structured} set we need to embed inside $E_{m}$. The sequence $\{d_{N,m}\}_{N\geq 1}$ will be constructed in Lemma \ref{lma:dmExist-V2}, whereas the other two are as in Chen's work \cite{Che16}.

\subsubsection{\underline{Choosing $\tau_{N,m}$ and $t_{N,m}$}}\label{sec:choices} For a fixed $1\leq m\leq k$, we choose
\begin{equation}\label{ineq1-0140126}  t_{N,m}=\psi(N)^{\alpha_{m}}\theta_{N,m},
\end{equation}
where $\theta_{N,m}$ satisfies
\begin{equation*}
    \theta_{N,m}\in\left[\frac{1}{4},\frac{1}{2}\right]\cup[2,4]
\end{equation*}
and
\begin{equation}\label{ineq2-0140126}  
\psi(N)^{\alpha_{m}}\log{(8\Psi(N))}\approx\prod_{i=1}^{N-1}\theta_{i,m}
\end{equation}
for $N\geq N_{0}$ large enough such that \eqref{ineq2-0140126} holds for all $1\leq m\leq k$. Similarly, choose
\begin{equation}\label{ineq3-0140126}  \tau_{N,m}=\psi(N)^{\alpha_{m}-\frac{\beta_{m}}{2}}\vartheta_{N,m},
\end{equation}
where $\vartheta_{N,m}$ satisfies
\begin{equation*}
    \vartheta_{N,m}\in\left[\frac{1}{4},\frac{1}{2}\right]\cup[2,4]
\end{equation*}
and
\begin{equation}\label{ineq4-0140126}  
\psi(N+1)^{\alpha_{m}}\log{(8\Psi(N+1))}\approx\prod_{i=1}^{N}\vartheta_{i,m}
\end{equation}
for $N\geq N_{0}$ large enough. As shown in \cite{Che16}, these choices are possible and they guarantee that
\begin{equation*}
    \frac{\tau_{1,m}\cdot\ldots\cdot \tau_{N,m}}{t_{1,m}\cdot\ldots\cdot t_{N,m}}\approx\Psi(N)^{-\frac{\beta_{m}}{2}}
\end{equation*}
for $N\geq N_{0}$, which suffices to imply the Fourier decay claimed in \eqref{decaymum-140126}. For $N<N_{0}$, simply choose $\tau_{N,m}=t_{N,m}=1$. These sequences clearly satisfy \eqref{cond3-090126}.

\subsubsection{\underline{Linear independence}}\label{sec:LI} In this section we introduce a new definition and state some properties that we will need for the construction of the multilinear Knapp-type example. We will use a construction similar to the one in \cite{Che16}, with $k$ measures built containing arithmetically structured sets interacting tangentially in the following sense.
\begin{definition} Let $M$ be a positive integer. We say that the set $\{d_m\}_{m=1}^k$ is $M$-linearly independent ($M$-LI) if for all $\ell_i\in(-M,M)\cap\mathbb{Z}$,
    \begin{equation*}
        \sum_{m=1}^k \ell_m d_m=0
    \end{equation*}
    if and only if $\ell_m=0$ for all $m=1,\ldots k$.
\end{definition}
The following gives a simple way of constructing $M$-LI sets of integers.

\begin{lemma}\label{lemma-LI-V2} If for each $m=2,\ldots,k$
\begin{equation}\label{MLIcondition-V2}
    d_{m}> M\left(\sum_{i< m}d_{i}\right)
\end{equation}
then the set $\{d_{m}\}_{m=1}^{k}$ is $M$-LI.
\end{lemma}
\begin{proof} Let $\sum_{m=1}^k \ell_{m}d_m = 0$ with  $\ell_m\in (-M,M)\cap\mathbb{Z}$. We must have $\ell_{k}=0$, otherwise
$$d_{k}=\left|\sum_{i< k}\frac{\ell_{i}}{\ell_{k}}d_{i}\right|\leq M\left(\sum_{i< k}d_{i}\right)<d_{k},$$
which contradicts \eqref{MLIcondition-V2}. Continue recursively to show that $\ell_{m}=0$ for all $1\leq m\leq k$.
\end{proof}

\subsubsection{\underline{Choosing $d_{N,m}$}} The next lemma guarantees the maximum number of distinct sums between the different rescaled arithmetic progressions, which can be interpreted as a manifestation of transversality in this context. As mentioned in the beginning of this section, there is one extra technical feature in the Knapp example presented in this manuscript compared to previous works. The common differences of the arithmetic progressions present in \cite{HL} and \cite{Che16} play no role in their arguments, whereas we need to carefully choose them to avoid losing information here.

\begin{lemma}\label{lma:dmExist-V2} For every $1\leq m\leq k$, let $\alpha_{m}$ and $\beta_{m}$ be given with
\begin{equation}
    0<\beta_{m}\leq\alpha_{m}\leq 1,
\end{equation}
\begin{equation}\label{eq:conditiona-b}
    \alpha_{j+1}-\frac{\beta_{j+1}}{2}\leq\alpha_{j}-\frac{\beta_{j}}{2},\quad\textnormal{for all} \quad 1\leq j\leq k-1,
\end{equation}
and
\begin{equation}\label{cond3-ajbj-14012026}
       \left(\alpha_{k}-\frac{\beta_{k}}{2}\right)+(k-1)\left(\alpha_{1}-\frac{\beta_{1}}{2}\right)<1.
    \end{equation}
Then there exist $N_0\in\mathbb{N}$ such that
\begin{enumerate}
    \item For all $N\geq N_{0}$, for
    \begin{equation*}
        \tau_{N}:=\max_{1\leq m\leq k}\tau_{N,m},
    \end{equation*}
    and
    \begin{equation*}
        M_{N}=\left\lceil \frac{1}{\sum_{m=1}^k \beta_m} \right\rceil(\tau_{N}-1)+1,
    \end{equation*}
    there exists an $M_{N}$-LI set $\{d_{N,m}\}_{m=1}^k$ satisfying \eqref{cond2-090126}.
\item For all $N\geq N_{0}$, the arithmetic progressions 
\begin{equation*}
    V_{N,m}= \left\{1, 1 + d_{N,m}, \cdots, 1 + (M_{N}-1)d_{N,m}\right\}
\end{equation*}
satisfy 
\begin{equation}\label{cond-sum-V2}
    |V_{N,1}+\ldots +V_{N,k}|=\prod_{m=1}^{k}|V_{N,m}|.
\end{equation}
\end{enumerate}
\end{lemma}

\begin{proof}
   Choose $d_{N,m}$ satisfying $d_{N,1}=2$ and

\begin{equation*}
    d_{N,m}=M_{N}\left(\sum_{j=1}^{m-1}d_{N,j}\right)+1.
\end{equation*}
That is, each $d_{N,m}$ is of the form
    \begin{equation*}
        d_{N,m} = \left(2M_{N}+1\right)\left(M_{N}+1\right)^{m-2}.
    \end{equation*}
By Lemma \ref{lemma-LI-V2}, $\{d_{N,m}\}_{m=1}^k$ is $M_N$-LI. To prove \eqref{cond2-090126}, it suffices to show that 
\begin{equation}\label{cond1-14012026}
    1+(\tau_{N,m}-1)d_{N,m}=1+(\psi(N)^{\alpha_{m}-\frac{\beta_{m}}{2}}\vartheta_{N,m}-1)\left(2M_{N}+1\right)\left(M_{N}+1\right)^{m-2}\leq\psi(N)
\end{equation}
for $N$ large enough. By \eqref{eq:conditiona-b}, and our choices of $\tau_{N,m}$ and $M_{N}$, \eqref{cond1-14012026} holds for $N>N_{0}$ large enough as long as
    \begin{equation*}
       \left(\alpha_{m}-\frac{\beta_{m}}{2}\right)+(m-1)\left(\alpha_{1}-\frac{\beta_{1}}{2}\right)<1
    \end{equation*}
for all $m=2,\ldots,k$, which follows from \eqref{cond3-ajbj-14012026}. 

To prove \eqref{cond-sum-V2}, note that there are at most $\prod_{m=1}^{k}|V_{N,m}|$ different sums in $V_{N,1}+\ldots +V_{N,k}$. If two sums in $V_{N,1}+\ldots +V_{N,k}$ coincide, i.e. if there are two different sets of coefficients $\{\ell_{j}\}_{j=1}^{k},\{\widetilde{\ell}_{j}\}_{j=1}^{k}\subseteq(-M,M)\cap\mathbb{Z}$ such that
$$\sum_{j=1}^{k}(1+\ell_{j}d_{N,j})=\sum_{j=1}^{k}(1+\widetilde{\ell_{j}}d_{N,j}),$$
then of course
$$\sum_{j=1}^{k}(\ell_{j}-\widetilde{\ell_{j}})d_{j}=0.$$

Since the set $\{d_{m}\}_{m=1}^{k}$ is $M$-LI, we must have $\ell_{j}-\widetilde{\ell_{j}}=0$, a contradiction. We conclude that any two sums in $V_{N,1}+\ldots +V_{N,k}$ are different, which immediately implies \eqref{cond-sum-V2}.  
\end{proof}

\begin{definition}
    For each $m=1,\ldots,k$, let $P_m$ be an arithmetic progression with ratio $d_m$. We say that the arithmetic progressions $\{P_m\}_{m=1}^k$ are $M$-linearly independent ($M$-LI) if the corresponding set of ratios $\{d_m\}_{m=1}^k$ is $M$-LI.
\end{definition}

\subsection{Construction of Cantor sets}\label{sec:CantorConstruction} Given $0<\beta_{m}\leq\alpha_{m}<1$, let $\{\tau_{N,m}\}_{N\in\mathbb{N}}$ and $\{t_{N,m}\}_{N\in\mathbb{N}}$ be the two sequences chosen in \eqref{ineq1-0140126} and \eqref{ineq3-0140126}.

Consider a sequence of sets $\{A_{N,m}\}_{N\in\mathbb{N}_0}$ defined in the following way
\begin{align*}
    A_{0,m} &= \{0\},\\
    A_{N,a,m} &\subseteq [\psi(N)]/\Psi(N),\\
    A_{N,m} &= \bigcup_{a\in A_{N-1,m}} ( a + A_{N,a,m}),\\
    |A_{N,a,m}|&=t_{N,m}.
\end{align*}
Define also
\begin{equation*}
   \displaystyle E_{N,m} = \bigcup_{a\in A_{N,m}} a + [0,1/\Psi(N)], \quad E_m = \bigcap_{N=1}^\infty E_{N,m}.
\end{equation*}
$E_m$ is a Cantor-type set with a natural probability measure $\mu_m$ defined as the weak limit of 
$$\frac{\d\mu_{N,m}}{\d x} = \frac{1}{|E_{N,m}|}1_{E_{N,m}}(t)\,\mathrm{d}t.$$

We now continue with the construction by identifying the arithmetically structured sets $F_{N,m}\subseteq E_{N,m}$. For each $N$ large enough, let $\{d_{N,m}\}_{m=1}^k$ be the $\Big(\left\lceil \frac{1}{\sum_{m=1}^k \beta_m} \right\rceil(\tau_{N}-1)+1\Big)$-LI set given by Lemma \ref{lma:dmExist-V2}. For each $m=1,\ldots,k$ let $W_{N,m}\subseteq [\psi(N)]$ be the following arithmetic progression of length $\tau_{N,m}$ and common difference $d_{N,m}$:
\begin{equation*}
    W_{N,m} := \{1, 1 + d_{N,m}, \ldots, 1 + (\tau_{N,m}-1)d_{N,m}\}.
\end{equation*}
Define the sequences of sets $P_{N,m}\subseteq A_{N,m}$, $F_{N,m}\subseteq E_{N,m}$ as
\begin{align*}
    P_{0,m}&=\{0\},\\
    P_{N+1,m} &= \bigcup_{a\in P_{N,m}} (a + 
    W_{N,m}/\Psi(N)),\\
    F_{N,m} &= \bigcup_{a\in P_{N,m}} a + [0,1/\Psi(N)).
\end{align*}
Note that $F_{N,m}\subseteq E_{N,m}$ are smaller Cantor-type sets with endpoints in a rescaled arithmetic progression. Finally, define for each $\ell\in\mathbb{N}$
\begin{equation*}
    f_{\ell,m} = 1_{F_{\ell,m}}.
\end{equation*}
These functions will play the role that the characteristic function of a small cap plays in the standard Knapp example.

The construction above is similar to the one given in \cite{Che16}, with the additional steps of starting the construction of the sets $P_{N,m}$ with an arithmetic progression of a given common difference $d_{N,m}$, and choosing them satisfying an appropriate linear independence property. The following proposition ensures that the measures $\mu_{N,m}$ have the properties that we need. 
\begin{proposition}\label{prop:propertiesChen}
Let $\phi:[2,\infty)\to(0,\infty)$ be a non-decreasing function with $\lim_{t\to\infty}\phi(t)=\infty$ and $\phi(2t)\leq\phi(t)+C$. For each $m=1,\ldots k$ there is a choice $A_{N,m}$ for every $N\in\mathbb{N}$ with the above properties such that for $0<\beta_m\leq\alpha_m$,
    \begin{equation}\label{decaymum-140126}
        \big|\widehat{\d\mu_m}(\xi)\big| \lesssim |\xi|^{-\beta_m/2}
    \end{equation}
    for all $\xi\in\mathbb{R}\backslash\{0\}$, and for all intervals $I$ centred in $E_m$ with $|I|<1/2$,    \begin{equation}\label{eq:ChensAhlfors}
        \frac{|I|^{\alpha_m}}{\phi(1/|I|)\log(1/|I|)} \lesssim \mu(I) \lesssim \frac{|I|^{\alpha_m}}{\log(1/|I|)}.
    \end{equation}
    What is more, if $\beta_m<\alpha_m$ the measure $\mu_m$ satisfies
    \begin{equation*}
        \frac{|I|^{\alpha_m}}{\phi(1/|I|)} \lesssim \mu(I) \lesssim |I|^{\alpha_m}
    \end{equation*}
    instead of \eqref{eq:ChensAhlfors}.
\end{proposition}
\begin{proof}
    The differences in the construction of Section~\ref{sec:CantorConstruction} do not affect the results in Section~4.4 and 4.5 of \cite{Che16}. Indeed, the key estimate in \cite{Che16} to prove a Fourier decay inequality like our \eqref{decaymum-140126} depends only on the parameters $\tau_{N,m}$ and $t_{N,m}$, which we choose in the same way. As for \eqref{eq:ChensAhlfors}, it depends only on the choice of $t_{N,m}$. 
\end{proof}

\subsection{Proof of Theorem \ref{nec-sing}}\label{sec:proofProp} We start by restating Theorem~\ref{nec-sing} with more details for clarity.
\begin{restatetheorem}{nec-sing} For each $m=1,\ldots,k$ let $0\leq\beta_m\leq\alpha_m \leq 1$ be such that
\begin{equation}
    \alpha_{j+1}-\frac{\beta_{j+1}}{2}\leq\alpha_{j}-\frac{\beta_{j}}{2},\quad\textnormal{for all} \quad 1\leq j\leq k-1,
\end{equation}
and
\begin{equation}
       \left(\alpha_{k}-\frac{\beta_{k}}{2}\right)+(k-1)\left(\alpha_{1}-\frac{\beta_{1}}{2}\right)<1.
    \end{equation}
Let $\mu_m$ be the measures defined above, where the corresponding arithmetic progressions $\{W_{N,m}\}_{m=1}^k$ are $\Big(\left\lceil \frac{1}{\sum_{m=1}^k \beta_m} \right\rceil(\tau_{N}-1)+1\Big)$-LI for each $N$ large enough. Let $\{f_{\ell,m}\}_{\ell\in\mathbb{N}} = \{1_{F_{\ell,m}}\}_{\ell\in\mathbb{N}}$. Assume that $p,q\in[1,\infty]$ satisfy
    \begin{equation*}
        q < \frac{2p(1-\sum_{m=1}^k\alpha_m) + p\sum_{m=1}^k\beta_m}{(p-1)\sum_{m=1}^{k}\beta_m}
    \end{equation*}
    Then
    \begin{equation*}
    \frac{\bigg\|\prod_{m=1}^{k} \reallywidehat{f_{\ell,m}\d\mu_m}\bigg\|_{L^q(\mathbb{R})}}{\prod_{m=1}^{k}\|f_{\ell,m}\|_{L^p(\d\mu_m)}} \to\infty \text{ as }\ell\to\infty.
    \end{equation*}
\end{restatetheorem}
\begin{proof}
Throughout this proof we will assume that $N$ is large enough as needed in Lemma~\ref{lma:dmExist-V2} and subsection~\ref{sec:choices}. First note that 
\begin{equation*}
\|\widehat{f_{\ell,m}\d\mu_m}\|_{L^\infty(\mathbb{R})} = \mu_m(F_{\ell,m}) = \prod_{i=1}^\ell \frac{\tau_{i,m}}{t_{i,m}}.
\end{equation*}
    Now we continue with some reductions. Let $r=\lceil \frac{1}{\sum_{m=1}^k \beta_m}\rceil$, then for $1\leq q\leq 2r$,
    \begin{equation} \label{eq:LqLrnorms}
        \eqalign{
        \displaystyle\Bigg\|\prod_{m=1}^{k}\widehat{f_{\ell,m} \d\mu_{m}}\Bigg\|_{L^{q}(\mathbb{R})}^{q}&\displaystyle\geq \frac{ \|\prod_{m=1}^{k}\widehat{f_{\ell,m}\d\mu_{m}}\|_{L^{2r}(\mathbb{R})}^{2r}}{\prod_{m=1}^{k}\| \widehat{f_{\ell,m}\d\mu_{m}}\|_{L^\infty(\mathbb{R})}^{2r-q}} \cr
        &\displaystyle\gtrsim \prod_{i=1}^\ell \prod_{m=1}^k \Big(\frac{t_{i,m}}{\tau_{i,m}}\Big)^{2r-q}  \Bigg\|\prod_{m=1}^{k}\widehat{f_{\ell,m} \d\mu_{m}}\Bigg\|_{L^{2r}(\mathbb{R})}^{2r}.
        }
    \end{equation}
    The strategy will be to use the fact that $2r\in 2\mathbb{Z}$ and find good lower bounds for 
    \begin{equation*}
        \left\|\prod_{m=1}^{k} \widehat{f_{\ell,m}\d\mu_{m}}\right\|_{L^{2r}(\mathbb{R})}
    \end{equation*}
by expanding the $L^{2r}(\mathbb{R})$ norm. Since $r>\frac{1}{\sum_{m=1}^k \beta_m}$, by Proposition~\ref{prop:propertiesChen} $\big|\widehat{f_{\ell,m}\d\mu_{N,m}}(\xi)\big|\lesssim g\in L^{2r}(\mathbb{R})$, and since $\widehat{f_{\ell,m}\d\mu_{N,m}} \to \widehat{f_{\ell,m}\d\mu_{m}}$ pointwise by the Portmanteau's theorem, we conclude that
\begin{equation*}
    \|\widehat{f_{\ell,m}\d\mu_{N,m}}\|_{L^{2r}(\mathbb{R})}\rightarrow\|\widehat{f_{\ell,m}\d\mu_{m}}\|_{L^{2r}(\mathbb{R})},\quad\textnormal{as }N\rightarrow\infty.
\end{equation*}
Arguing as in \cite[Section~4.6]{Che16}, we know that for each $m=1,\ldots,k$,
\begin{equation*}
    \|f_{\ell,m}\|_{L^{p}(\d\mu_{N,m})} = \prod_{i=1}^\ell \Big( \frac{\tau_{i,m}}{t_{i,m}} \Big)^{1/p}.
\end{equation*}
Let $\mathcal{A}_{\ell,N,m} = F_{\ell,m}\cap A_{N,m}$. Recall that the $A_{N,m}$ are finite, therefore so are the $\mathcal{A}_{\ell,N,m}$
\begin{equation*}
    f_{\ell,m}\d\mu_{N,m} = \frac{1}{t_{1,m}\ldots t_{N,m}} \Big( \sum_{a\in \mathcal{A}_{\ell,N,m}} \delta_a \Big)*\Psi(N)1_{[0,1/\Psi(N))}.
\end{equation*}
Therefore, by developing the $L^{2r}$ norms and by changing variables,
\begin{align*}
    & \Bigg\|\prod_{m=1}^{k}\widehat{f_{\ell,m}\d\mu_{N,m}}\Bigg\|_{L^{2r}(\mathbb{R})}^{2r} \\
  &= \left(\prod_{m=1}^{k} \frac{1}{t_{1,m}\cdots t_{N,m}}\right)^{2r} \int_{\mathbb{R}} \sinc\Big( \frac{\xi}{\Psi(N)} \Big)^{2kr} \Bigg| \sum_{a_{1}\in \mathcal{A}_{\ell,N,1}}\cdots\sum_{a_{k}\in \mathcal{A}_{\ell,N,k}}  e^{-2\pi i \xi(a_{1}+\cdots+a_{k})} \Bigg|^{2r} \,\d\xi \\
  &= \left(\prod_{m=1}^{k} \frac{1}{t_{1,m}\cdots t_{N,m}}\right)^{2r} \int_{\mathbb{R}} \sinc\Big( \frac{\xi}{\Psi(N)} \Big)^{2kr} \\
  &\qquad\times\left( \sum_{\substack{a_{n,1}\in \mathcal{A}_{\ell,N,1} \\ 1\leq n\leq 2r }}\cdots\sum_{\substack{a_{n,k}\in \mathcal{A}_{\ell,N,k} \\ 1\leq n\leq 2r }} e^{-2\pi i \xi\Big( \sum_{n=1}^r\sum_{m=1}^{k} (a_{n,m}-a_{n+r,m})\Big)}\right) \,\d\xi \\
  &= \left(\prod_{m=1}^{k} \frac{1}{t_{1,m}\cdots t_{N,m}}\right)^{2r} \sum_{\substack{a_{n,1}\in \mathcal{A}_{\ell,N,1} \\ 1\leq n\leq 2r }}\cdots\sum_{\substack{a_{n,k}\in \mathcal{A}_{\ell,N,k} \\ 1\leq n\leq 2r }}  \int_{\mathbb{R}}  \sinc\Big( \frac{\xi}{\Psi(N)} \Big)^{2kr}  \\
  &\qquad \times e^{-2\pi i \xi\Big( \sum_{n=1}^r\sum_{m=1}^{k} (a_{n,m}-a_{n+r,m})\Big)} \,\d\xi \\
   &= \left(\prod_{m=1}^{k} \frac{1}{t_{1,m}\cdots t_{N,m}}\right)^{2r} \Psi(N)\sum_{\substack{a_{n,1}\in \Psi(N)\mathcal{A}_{\ell,N,1} \\ 1\leq n\leq 2r }}\cdots\sum_{\substack{a_{n,k}\in \Psi(N)\mathcal{A}_{\ell,N,k} \\ 1\leq n\leq 2r }}  \int_{\mathbb{R}}  \sinc(\eta)^{2kr}  \\
  &\qquad \times e^{-2\pi i \eta\Big( \sum_{n=1}^r\sum_{m=1}^{k}( a_{n,m}-a_{n+r,m})\Big)} \,\d\eta.
\end{align*}
Let $K(\eta) = \sinc(\eta)^{2kr}$, then
\begin{align*}
    \Bigg\|\prod_{m=1}^{k}\widehat{f_{\ell,m} \d\mu_{N,m}}\Bigg\|_{L^{2r}(\mathbb{R})}^{2r} &= \left(\prod_{m=1}^{k} \frac{1}{t_{1,m}\cdots t_{N,m}}\right)^{2r}\Psi(N) \\
    &~\quad\times\sum_{\substack{a_{n,1}\in \Psi(N)\mathcal{A}_{\ell,N,1} \\ 1\leq n\leq 2r }}\cdots\sum_{\substack{a_{n,k}\in \Psi(N)\mathcal{A}_{\ell,N,k} \\ 1\leq n\leq 2r }} \widehat{K}\Big( \sum_{n=1}^r\sum_{m=1}^{k} (a_{n,m}-a_{n+r,m} )\Big).
\end{align*}

The following is the first step at which we potentially lose some information. We bound this sum from below by only keeping the terms that give $\widehat{K}(0)$, thus
\begin{equation*}
    \Bigg\|\prod_{m=1}^{k}\widehat{f_{\ell,m} \d\mu_{N,m}}\Bigg\|_{L^{2r}(\mathbb{R})}^{2r} \geq \left(\prod_{m=1}^{k} \frac{1}{t_{1,m}\cdots t_{N,m}}\right)^{2r}\Psi(N)\sum_{\substack{a_{n,1}\in \Psi(N)\mathcal{A}_{\ell,N,1}\\
    \vdots
    \\
    a_{n,k}\in \Psi(N)\mathcal{A}_{\ell,N,k} \\ 
    1\leq n\leq 2r\\
    \sum_{n=1}^r\sum_{m=1}^{k} a_{n,m}-a_{n+r,m}=0}}  \widehat{K}(0).
\end{equation*}

Let $C(2kr) = \widehat{K}(0) = \int_{\mathbb{R}}\sinc{(\eta)}^{2kr}\d\eta$. This gives
\begin{align}\label{ineq1-220126}
 \Bigg\|\prod_{m=1}^{k}\widehat{f_{\ell,m}\d\mu_{N,m}}\Bigg\|_{L^{2r}(\mathbb{R})}^{2r} &\geq C(2kr)\Psi(N)\left(\prod_{m=1}^{k}\frac{1}{t_{1,m}\cdots t_{N,m}}\right)^{2r}\\
 &~\quad \times \Bigg| \Bigg\{ \parbox{27em}{\begin{center}
      $(a_{1,1},\ldots,a_{2r,1}, \cdots,a_{1,k},\ldots,a_{2r,k})\in \prod_{m=1}^{k}(\mathcal{A}_{\ell,N,m})^{2r}$ \\
      such that\\$\sum_{n=1}^r\sum_{m=1}^{k} a_{n,m}=\sum_{n=1}^r\sum_{m=1}^{k}a_{n+r,m}$
    \end{center}} \Bigg\}\Bigg|
\end{align}
by homogeneity of the relation $\sum_{n=1}^r\sum_{m=1}^{k} a_{n,m}=\sum_{n=1}^r\sum_{m=1}^{k}a_{n+r,m}$. For each $\ell,N\in\mathbb{N}$ define
\begin{equation*}
    M_{\ell,N,r}=\Bigg|\Bigg\{ \parbox{25em}{\begin{center}
      $(a_{1,1},\ldots,a_{2r,1}, \cdots,a_{1,k},\ldots,a_{2r,k})\in \prod_{m=1}^{k} (\mathcal{A}_{\ell,N,m})^{2r}$ \\
      such that\\$\sum_{n=1}^r\sum_{m=1}^{k} a_{n,m}=\sum_{n=1}^r\sum_{m=1}^{k}a_{n+r,m}$
    \end{center}} \Bigg\}\Bigg|.
\end{equation*}
We want to bound $M_{\ell,N,r}$ from below. For this, define for each $\ell,N\in\mathbb{N}$,
\begin{equation*}
    Z_{\ell,N,r}=\left\{\sum_{n=1}^{r}\sum_{m=1}^{k}a_{n,m} : a_{n,m}\in \mathcal{A}_{\ell,N,m}, \quad\forall~ 1\leq n\leq r\right\}.
\end{equation*}
Observe that $y_m\in \mathcal{A}_{\ell,N,m}$ has a digit representation
\begin{equation*}
    y_m = \sum_{i=1}^\ell \frac{y_m^{(i)}}{\Psi(i)} + \frac{y_m^{(\ell+1)}}{\Psi(N)},
\end{equation*}
where $y_m^{(i)}\in W_{i,m}$ for $i=1,\ldots,\ell$ and $W_{i,m} = \{1, 1 + d_{i,m},\ldots,1 + (\tau_{N,m}-1)d_{N,m}\}$ with $\{d_{N,m}\}_{m=1}^k$ the $\Big(\left\lceil \frac{1}{\sum_{m=1}^k \beta_m} \right\rceil(\tau_{N}-1)+1\Big)$-LI set; and $y_m^{(\ell+1)}\in [\psi(\ell+1)\cdots\psi(N)]$. Then, each $z\in Z_{\ell,N,r}$ can be written as 
\begin{equation*}
    z = \sum_{i=1}^\ell \frac{z^{(i)}}{\Psi(i)} + \frac{z^{(\ell+1)}}{\Psi(N)}
\end{equation*}
with $z^{(i)}\in W_{i,1}' +\cdots + W_{i,k}'$ for $i=1,\ldots,\ell$ and $W_{j,m}' = \{r, r + d_{i,m}, \ldots, r+ r(\tau_{i,m}-1)d_{i,m}\}$; and $z^{(\ell+1)}\in[kr\psi(\ell+1)\cdots \psi(N)]$.

Therefore, by Lemma~\ref{lma:dmExist-V2} with $V_{N,m}=W_{N,m}'$,
\begin{align*}
    |Z_{\ell,N,r}| &\leq |W_{1,1}' + \cdots + W_{1,k}'|\cdots |W_{\ell,1}' + \cdots + W_{\ell,k}'|\,|[r\psi(\ell+1)\cdots \psi(N)]|\\
    &=\prod_{i=1}^\ell \prod_{m=1}^k \big(r(\tau_{i,m}-1) +1\big) (r\psi(\ell+1)\cdots \psi(N)).
\end{align*}
The inequality above comes exclusively from saying the set of possible $z^{(\ell+1)}$ is in $[kr\psi(\ell+1)\cdots \psi(N)]$, which is the second potentially lossy step. \textbf{We do not lose any information with the terms} $|W_{i,1}' + \cdots + W_{i,k}'|$, $1\leq i\leq\ell$, because we compute them explicitly. For $z\in Z_{\ell,N,r}$, define
\begin{equation*}
    g(z) = \Bigg|\Bigg\{(a_{1,1},\ldots,a_{1,r},\cdots,a_{k,1},\ldots,a_{k,r})\in \prod_{m=1}^{k} (\mathcal{A}_{\ell,N,m})^{r}: \sum_{n=1}^r\sum_{m=1}^{k} a_{n,m}=z\Bigg\}\Bigg|.
\end{equation*}
Then $\|g\|_{\ell^1(Z_{\ell,N,r})} = \prod_{m=1}^k |\mathcal{A}_{\ell,N,m}|^r = \prod_{m=1}^k (\tau_{1,m}\cdots\tau_{\ell,m} \,t_{\ell+1,m}\cdots t_{N,m})^{r}$ and
\begin{align*}
   &\displaystyle\|g\|_{\ell^2}^2 \cr
   &= \sum_{z}\left|\left\{(a_{1,1},\ldots,a_{1,r},\cdots,a_{k,1},\ldots,a_{k,r})\in \prod_{m=1}^{k} (\mathcal{A}_{\ell,N,m})^{r}: \sum_{n=1}^r\sum_{m=1}^{k} a_{n,m}=z\right\}\right|^{2} \cr
   &=\sum_{z}\left|\left\{(a_{1,1},\ldots,a_{1,2r},\cdots,a_{k,1},\ldots,a_{k,2r})\in \prod_{m=1}^{k} (\mathcal{A}_{\ell,N,m})^{2r}: \sum_{n=1}^r\sum_{m=1}^{k} a_{n,m}=\sum_{n=1}^r\sum_{m=1}^{k}a_{n+r,m}=z\right\}\right| \cr
   &= M_{\ell,N,r}.
\end{align*}

This way, by Cauchy--Schwarz (which is the third potentially lossy step),
\begin{equation*}
    M_{\ell,N,r} \geq \frac{\|g\|_{\ell^1(Z_{\ell,N,r})}^2}{|Z_{\ell,N,r}|} \geq \frac{\prod_{m=1}^k (\tau_{1,m}\cdots\tau_{\ell,m} \,t_{\ell+1,m}\cdots t_{N,m})^{2r}}{\prod_{i=1}^\ell \prod_{m=1}^k \big(r(\tau_{i,m}-1) +1\big) (r\psi(\ell+1)\cdots \psi(N))}.
\end{equation*}
Expanding $\prod_{i=1}^\ell \prod_{m=1}^k \big(r(\tau_{i,m}-1) +1\big)$,
\begin{equation*}
    \prod_{i=1}^\ell \prod_{m=1}^k \big(r(\tau_{i,m}-1) +1\big) = r^{k\ell} \prod_{i=1}^\ell \prod_{m=1}^k \tau_{i,m} + L_{<k},
\end{equation*}
where $L_{<k}$ is a lower order term given by a linear combination of products of less than $k$ factors of the form $r\tau_{i,m}$ and $r$. Thus,
\begin{equation*}
    M_{\ell,N,r} \geq \frac{\prod_{m=1}^k (\tau_{1,m}\cdots\tau_{\ell,m} \,t_{\ell+1,m}\cdots t_{N,m})^{2r}}{r^{k\ell} \prod_{i=1}^\ell \prod_{m=1}^k \tau_{i,m} + L_{<k}} \frac{\Psi(\ell)}{r\Psi(N)}.
\end{equation*}
Therefore, by \eqref{ineq1-220126},
\begin{align*}
    &\Bigg\|\prod_{m=1}^{k}\widehat{f_{\ell,m} \d\mu_{N,m}}\Bigg\|_{L^{2r}(\mathbb{R})}^{2r}\\
    &\geq C(2kr)\Psi(N)\left(\prod_{m=1}^{k}\frac{1}{t_{1,m}\cdots t_{N,m}}\right)^{2r}\frac{\prod_{m=1}^k (\tau_{1,m}\cdots\tau_{\ell,m} \,t_{\ell+1,m}\cdots t_{N,m})^{2r}}{r^{k\ell} \prod_{i=1}^\ell \prod_{m=1}^k \tau_{i,m} + L_{<k}} \frac{\Psi(\ell)}{r\Psi(N)} \\
    &=  C(2kr)\Psi(\ell)\frac{\prod_{i=1}^\ell\prod_{m=1}^k (\tau_{i,m}^{2r}t_{i,m}^{-2r})}{r^{k\ell+1} \prod_{i=1}^\ell \prod_{m=1}^k \tau_{i,m} + rL_{<k}}.
\end{align*}
From \eqref{eq:LqLrnorms}, for $1\leq q\leq 2r$,
\begin{align*}
    &\Bigg\|\prod_{m=1}^{k}\widehat{f_{\ell,m} \d\mu_{N,m}}\Bigg\|_{L^{q}(\mathbb{R})}^{q} \gtrsim \prod_{i=1}^\ell \prod_{m=1}^k \Big(\frac{t_{i,m}}{\tau_{i,m}}\Big)^{2r-q}  \Bigg\|\prod_{m=1}^{k}\widehat{f_{\ell,m} \d\mu_{m}}\Bigg\|_{L^{2r}(\mathbb{R})}^{2r} \\
    &\geq C(2kr)\Psi(\ell) \prod_{i=1}^\ell \prod_{m=1}^k \Big(\frac{t_{i,m}}{\tau_{i,m}}\Big)^{2r-q}  \frac{\prod_{i=1}^\ell\prod_{m=1}^k (\tau_{i,m}^{2r}t_{i,m}^{-2r})}{r^{k\ell+1} \prod_{i=1}^\ell \prod_{m=1}^k \tau_{i,m} + rL_{<k}}.
\end{align*}
We conclude that
\begin{align*}
    \frac{\bigg\|\prod_{m=1}^{k}\widehat{f_{\ell,m} \d\mu_{N,m}}\bigg\|_{L^{q}(\mathbb{R})}^{q}}{\prod_{m=1}^k \|f_{\ell,m}\|_{L^{p}(\d\mu_{N,m})}^q} &\geq C(2kr)\Psi(\ell) \prod_{i=1}^\ell \prod_{m=1}^k \Big(\frac{t_{i,m}}{\tau_{i,m}}\Big)^{\frac{q}{p}-q } \Big( r^{k\ell+1} \prod_{i=1}^\ell \prod_{m=1}^k \tau_{i,m} + rL_{<k} \Big)^{-1} \\
    &= \frac{C(2kr) \Psi(\ell)}{r^{k\ell+1}\prod_{i=1}^\ell \prod_{m=1}^k (\tau_{i,m}^{1+\frac{q}{p}-q} t_{i,m}^{q-\frac{q}{p}}) + rL_{<k} \prod_{i=1}^\ell \prod_{m=1}^k  (\tau_{i,m}^{\frac{q}{p}-q}t_{i,m}^{q-\frac{q}{p}})}.
\end{align*}

Using the choices made in Section~\ref{sec:choices} yields
\begin{equation}\label{eq:limit}
    \displaystyle\frac{\bigg\| \prod_{m=1}^{k}\widehat{f_{\ell,m} \d\mu_{N,m}}\bigg\|_{L^{q}(\mathbb{R})}^{q}}{\prod_{m=1}^k \|f_{\ell,m}\|^q_{L^{p}(\d\mu_{N,m})}}\gtrsim \Gamma(\ell)^{-1}, 
\end{equation}
where
\begin{equation*}
\eqalign{
   \displaystyle \Gamma(\ell):=&\displaystyle r^{k \ell+1}\Psi(\ell)^{-1+\sum_{m=1}^k \alpha_m - \frac{1}{2}\big( 1+ \frac{q}{p}-q \big)\sum_{m=1}^k \beta_m} \psi(\ell+1)^{\sum_{m=1}^k\alpha_m}\log^k(8\Psi(\ell+1)) \cr
   &\qquad\displaystyle + r\Psi(\ell)^{-1}L_{<k} \prod_{i=1}^\ell \prod_{m=1}^k  (\tau_{i,m}^{\frac{q}{p}-q}t_{i,m}^{q-\frac{q}{p}}). \cr
    }
\end{equation*}
Because of our choices of $\psi$ and $\Psi$, the right-hand side of \eqref{eq:limit} diverges as $\ell\to\infty$ if 
\begin{equation*}
    1 - \sum_{m=1}^k \alpha_m + \frac{1}{2}\sum_{m=1}^k \beta_m \geq \frac{q}{2} \Big( 1 - \frac{1}{p}\Big) \sum_{m=1}^k \beta_m,
\end{equation*}
which results in $q < \frac{2p(1-\sum_{m=1}^k\alpha_m) + p\sum_{m=1}^k\beta_m}{(p-1)\sum_{m=1}^{k}\beta_m}$, as we wanted to show.
\end{proof}

\subsection{Second construction of Cantor sets}
As in the previous sections, the construction below will be a particular example of the one in \cite{HL}, with the additional step of considering $M$-LI APs, which can be thought of as a manifestation of transversality between the measures.

Let $N_0$ be an integer. For each $m=1,\ldots,k$ let $1<t_{0,m}< N_0$ be an integer such that for $\alpha_m= \log t_{0,m}/\log N_0$, $(k-1)\alpha_1 + \alpha_k<2$. Let $N = N_0^{2n_0}$, $t_m = t_{0,m}^{2n_0}$ for some $n_0$ large enough to be chosen later.

For each $m=1,\ldots,k$, consider a sequence of sets $\{A_{j,m}\}_{j\in\mathbb{N}_0}$ such that
\begin{align*}
    A_{0,m} &= \{0\},\\
    A_{j+1,a,m} &\subseteq N^{-(j+1)} [N],\\
    A_{j+1,m} &= \bigcup_{a\in A_{j,m}} ( a + A_{j+1,a,m}),\\
    |A_{j+1,a,m}|&=t_m.
\end{align*}
Define also
\begin{equation*}
   \displaystyle E_{j,m} = \bigcup_{a\in A_{j,m}} a + [0,N^{-j}], \quad E_m = \bigcap_{j=1}^\infty E_{j,m}.
\end{equation*}
$E_m$ is a Cantor-type set with a natural probability measure $\mu_m$ defined as the weak limit of 
$$\frac{\d\mu_{j,m}}{\d x} = \sum_{a\in A_{j,m}} N^{-j\alpha_m}N^j1_{[a,a+N^{-j}]}.$$

By a standard argument (see Lemma 6.1 of \cite{LM}), the sets $E_m$ defined above satisfy $\hd E_m = \alpha_m$ and each probability measure $\mu_m$ satisfies
\begin{equation*}
    \mu_m\big( B(x,r)\big) \lesssim_{n_0} r^{\alpha_m}
\end{equation*}
for any $x\in E_m$ and $r>0$.

We now continue with the construction by identifying the arithmetically structured sets $F_{j,m}\subseteq E_{j,m}$. Since these will be given by $M$-LI APs, we need the following lemma that plays the same role as Lemma~\ref{lma:dmExist-V2}.
\begin{lemma}\label{lma:dmExist} Let $0<\alpha_k<\cdots < \alpha_1<1$,  with $\alpha_{m}=\frac{\log{t_{0,m}}}{\log{N_{0}}}$ for some integers $t_{0,m}$ and $N_{0}$, $t_{0,m}\leq N_{0}$ for each $1\leq m\leq k$, and $(k-1)\alpha_1 + \alpha_k <2$. Then
\begin{enumerate}
    \item There exist $n_0\in\mathbb{N}$ and a set $\{d_m\}_{m=1}^k$ satisfying
    \begin{equation}\label{cond-dm}
        d_m \leq N_0^{2n_0\big(1-\frac{\alpha_m}{2}\big)}
    \end{equation}
$m=1,\ldots,k$, that is $\Big(\left\lceil\frac{1}{\sum_{m=1}^k \alpha_m}\right\rceil (N_0^{n_0\alpha_1}-1)+1\Big)$-LI.
\item For any $a_{i}\in\mathbb{Z}$, the arithmetic progressions 
\begin{equation*}
    V_m= \left\{a_{m}, a_{m} + d_m, \cdots, a_{m} + \left\lceil\tfrac{1}{\sum_{m=1}^{k}\alpha_m}\right\rceil( N_0^{n_0\alpha_1}-1)d_m\right\}
\end{equation*}
$m=1,\ldots,k$, satisfy 
\begin{equation}\label{cond-sum}
    |V_1+\ldots +V_{k}|=\prod_{m=1}^{k}|V_{m}|.
\end{equation}
\end{enumerate}
\end{lemma}

Let $\{d_m\}_{m=1}^k$ be the $\Big(\left\lceil\frac{1}{\sum_{m=1}^k \alpha_m}\right\rceil (N_0^{n_0\alpha_1}-1)+1\Big)$-LI set given by Lemma~\ref{lma:dmExist}. For each $m=1,\ldots,k$ let $P_m\subseteq [N]$ be an arithmetic progression of length $t_m^{1/2}=N^{\alpha_m/2}$ and ratio $d_m$, that is
\begin{equation*}
    P_m = \{x_{m}, x_{m} + d_{m}, x_{m} + 2d_{m}, \ldots, x_{m} + (N^{\frac{\alpha_m}{2}}-1)d_{m}\} 
\end{equation*}
for some $x_{m}\in [N]$. Define the sequences of sets $P_{j,m}\subseteq A_{j,m}$, $F_{j,m}\subseteq E_{j,m}$ as
\begin{align*}
    P_{0,m}&=\{0\},\\
    P_{j+1,m} &= \bigcup_{a\in P_{j,m}} (a + N^{-(j+1)}P_m),\\
    F_{j,m} &= \bigcup_{a\in P_{j,m}} a + [0,N^{-j}).
\end{align*}
Note that $F_{j,m}\subseteq E_{j,m}$ are smaller Cantor-type sets with endpoints in a rescaled arithmetic progression. 

We now restate Theorem~\ref{HL-nec-sing} in greater detail. Its proof is a routine modification of the proof of Theorem~\ref{nec-sing}.
\begin{restatetheorem}{HL-nec-sing} For each $m=1,\ldots,k$ let $0<\alpha_k<\cdots <\alpha_{1}<1$ such that $\alpha_m = \frac{\log t_{0,m}}{\log N_0}$ for some $t_{0,m},N_0\in\mathbb{N}$, and $(k-1)\alpha_1 + \alpha_k <2$. Let $\mu_m$ be the measures defined above, where the corresponding arithmetic progressions $\{P_{m}\}_{m=1}^k$ are $\Big(\left\lceil\frac{1}{\sum_{m=1}^k \alpha_m}\right\rceil (N_0^{n_0\alpha_1}-1)+1\Big)$-LI. Let $\{f_{\ell,m}\}_{\ell\in\mathbb{N}} = \{1_{F_{\ell,m}}\}_{\ell\in\mathbb{N}}$. Assume that $p,q\in[1,\infty]$ satisfy
    \begin{equation*}
        q < \frac{p(2-\sum_{m=1}^k\alpha_m)}{(p-1)\sum_{m=1}^{k}\alpha_m}.
    \end{equation*}
    Then
    \begin{equation*}
    \frac{\bigg\|\prod_{m=1}^{k} \reallywidehat{f_{\ell,m}\d\mu_m}\bigg\|_{L^q(\mathbb{R})}}{\prod_{m=1}^{k}\|f_{\ell,m}\|_{L^p(\d\mu_m)}} \to\infty \text{ as }\ell\to\infty.
    \end{equation*}
\end{restatetheorem}

\end{document}